\documentclass[11pt]{amsart}
\usepackage{amsmath}
\usepackage{cases}
\usepackage{mathrsfs}
\usepackage{bbm}
\usepackage{amssymb}
\usepackage{txfonts}
\usepackage{amscd}
\usepackage{amsfonts,latexsym,amsmath,amsthm,amsxtra,mathdots,amssymb,latexsym,mathabx}
\usepackage[all,cmtip]{xy}
\usepackage{cite}
\RequirePackage{amsmath} \RequirePackage{amssymb}
\usepackage{color}
\usepackage{colordvi}
\usepackage{multicol}
\usepackage{hyperref}
\usepackage[toc,page]{appendix}

\usepackage[margin=1in]{geometry}

\usepackage{xcolor}
\hypersetup{
    colorlinks,
    linkcolor={red!50!black},
    citecolor={blue!50!black},
    urlcolor={blue!80!black}
}

\def\mod{\mathrm{mod}\ }

\def\sumx{\sideset{}{^\star}\sum}

    \newcommand{\BC}{{\mathbb {C}}} 
     \newcommand{\BF}{{\mathbb {F}}}

    \newcommand{\BQ}{{\mathbb {Q}}} \newcommand{\BR}{{\mathbb {R}}}

     \newcommand{\BZ}{{\mathbb {Z}}}

     \newcommand{\CJ}{{\mathcal {J}}}

    \newcommand{\CQ}{{\mathcal {Q}}}

\newcommand{\units}{\mathcal{O}_K^\times}
  
\newcommand{\R}{\mbox{$\mathbb R$}}	

\def\fra{\mathfrak{a}}

\def\frd{\mathfrak{d}}
\def\frg{\mathfrak{g}}

\def\frp{\mathfrak{p}}

\def\-{^{-1}}

\def\frl{\mathfrak{L}}

\def\calO{\mathcal{O}}

\def\calN{\mathcal{N}}
\def\calA{\mathcal{A}}
\def\calB{\mathcal{B}}

\def\-{^{-1}}

\newcommand{\delete}[1]{}

    \theoremstyle{plain}

    \newtheorem{thm}{Theorem}[section] \newtheorem{cor}[thm]{Corollary}
    \newtheorem{lem}[thm]{Lemma}  \newtheorem{prop}[thm]{Proposition}

    \newtheorem {rem}[thm]{Remark} 
    \newtheorem {notation}[thm]{Notation}
    
    \newcommand{\Gros}{Gr\"{o}ssencharakter~}
    \newcommand{\Lind}{Lindel\"{o}f~}
    
    \numberwithin{equation}{section}

\makeindex
\begin{document}

\title{Subconvexity Bound for Hecke character $L$-Functions of Imaginary quadratic Number fields}
\maketitle
\begin{center}
\author{Keshav Aggarwal}
\end{center}
\begin{abstract}
Let $K=\BQ(\sqrt{-D})$ be an imaginary number field, $(p)=\frp\frp'$ be a split odd prime and $\psi$ be a Hecke character of conductor $\frp$. Let $L(s,\psi)$ be the associated $L$-function. We prove the Burgess bound in $t$-aspect and a hybrid bound in conductor aspect,
\begin{equation*}
L(1/2+it,\psi)\ll_{D,\varepsilon} (1+|t|)^{3/8+\varepsilon}p^{1/8}
\end{equation*}
for $p\ll t$. In Appendix A, we present the ideas for an elementary proof of Voronoi summation formula for holomorphic cusp forms with CM and squarefree level. This is done by exploiting the lattice structure of ideals in number fields. Voronoi summation for such cusp forms is given by Kowalski, Michel and Vanderkam \cite{kmv}. We hope that our method of proof can extend their result to any CM cusp form in $S_k(\Gamma_1(N))$ and arbitrary additive twist. We encounter quadratic and quartic Gauss sums in the process. We shall present the calculations for the general case in the next version of the paper.
\end{abstract}

\tableofcontents

\section{Introduction}

Let $K=Q(\sqrt{-D})$ be an imaginary quadratic number field with $D>0$ squarefree. The units in the ring of integers $\calO_K$ are the roots of unity $\mu_K$. Let $p$ be an odd split prime, say $(p)=\frp\frp'$. By Dedekind's theorem, the polynomial $f(x)=x^2+D=(x+d)(x-d)\bmod p$ splits into linear factors. Let $\frp$ be generated by $\{p,\sqrt{-D}-d\}$ over $\calO_K$. A possible isomorphism $\Theta:\calO_K/\frp\rightarrow \BF_p$ is given by $\Theta(a+b\sqrt{-D})=a+bd\bmod p$ (when $-D\equiv2,3\bmod4$) or $\Theta(\frac{a+b\sqrt{-D}}{2})=(a+bd)\overline{2}\bmod p$ (when $-D\equiv1\bmod4$).

Let $\psi$ be a Hecke character of conductor $\frp$ and weight $r$. Then $\psi$ is given by,
\begin{equation*}
\begin{split}
\psi(a+b\sqrt{-D})&=\chi(a+bd)\left(\frac{a+b\sqrt{-D}}{\sqrt{a^2+b^2D}}\right)^r, \quad \textit{ when } -D\equiv2,3\bmod4\\
\psi\left(\frac{a+b\sqrt{-D}}{2}\right)&=\chi(a+bd)\overline{\chi}(2)\left(\frac{a+b\sqrt{-D}}{\sqrt{a^2+b^2D}}\right)^r, \quad \textit{ when } -D\equiv1\bmod4.
\end{split}
\end{equation*}
where $\chi$ is a primitive Dirichlet character $\bmod p$ and $\psi(\alpha)=1$ for $\alpha\in\mu_K$.

The $L$-function associated with the Hecke character $\psi$ is given by
\begin{equation*}
L(s,\psi)=\underset{\begin{subarray}{c}\fra \textit{ is integal }\\ \textit{ coprime with } \frp\end{subarray}}\sum\frac{\psi(\fra)}{(\calN\fra)^s} \quad \textit{ for }Re(s)>1.
\end{equation*}
Hecke gave a functional equation for this $L$-function and extended it meromorphically to all of the complex plane. 
Phragmen-\Lind principle implies that $L(1/2+it,\psi)\ll_{\varepsilon}(1+|t|)^{1/2+\varepsilon}p^{1/4+\varepsilon}$- the convexity bound. The aim of this paper is to prove the following result.

\begin{thm}\label{mainthm}
Let $K=\BQ(\sqrt{-D})$ be an imaginary number field with $D>0$ squarefree. Let $p$ be an odd split prime, $p=\frp\frp'$. Let $\psi$ be a Hecke character with weight $r>0$ and conductor prime $\frp$, and $L(s,\psi)$ be the associated $L$-function. Then we have
\begin{equation*}
L(1/2+it,\psi)\ll_{D,\varepsilon} (1+|t|)^{3/8+\varepsilon}p^{1/8+\varepsilon}
\end{equation*}
for $p\ll t$.
\end{thm}

The first such result was obtained by Kaufman \cite{kaufman79} and Sohne \cite{sohne97}, who obtained
\begin{equation*}
L(1/2+it,\psi)\ll (1+|t|)^{d/6+\varepsilon}.
\end{equation*}
Here $\psi$ is a Hecke character of an arbitrary number field $K$ and $d$ is the degree of $K/\BQ$. Sohne actually got a stronger result which implies subconvexity in level aspect when the conductor of $\psi$ has a small factor. Diaconu and Garrett \cite{diga10} proved a $t$-aspect subconvexity bound for $L$-functions attached to cuspforms $f$ for $GL_2(K)$. Their argument uses asymptotics with error term with a power
saving for second integral moments over spectral families of twists $L(s,f\otimes\psi)$ by Hecke characters $\psi$. Michel and Venkatesh \cite{mive10} used spectral theory and gave a subconvexity bound uniformly in all aspects for $L$-functions attached to cusp forms $f$ for $GL_2(K)$ and their $GL(1)$ twists $L(s,f\otimes\psi)$. Wu \cite{wu16} followed \cite{mive10} and used an amplification technique to get a subconvexity bound uniformly in both conductor and $t$-aspect,
\begin{equation*}
L(1/2+it,\psi)\ll_\varepsilon \CQ^{1/4-(1-2\theta)/16+\varepsilon}.
\end{equation*}
Here $\theta$ is any bound towards the Ramanujan-Petersson conjecture. Very recently, Booker, Milinovich and Ng \cite{boming17} proved a $t$-aspect subconvexity bound for $L$-functions of cusp forms $f\in S_k(\Gamma_1(N))$,
\begin{equation*}
L(1/2+it,f)\ll (1+|t|)^{1/3}\log|t|.
\end{equation*}
They thus obtain a bound of strength comparable to Good's bound for the full modular group. A key innovation in their proof is a general form of Voronoi summation that applies to all fractions, even when the level is not squarefree.

We also require a Voronoi summation formula for any fraction and any level. We worked in parallel to find a way of proof of Voronoi summation formula for holomorphic cusp forms with complex multiplication in $S_k(\Gamma_1(N))$. This is discussed in Appendix A. This also happens to be the key innovation in our proof. Though our bound is not as strong as that of \cite{boming17}, we hope to be able to achieve that by more careful analysis as done in \cite{singh17}. Moreover, our result gives a hybrid bound in level aspect. We must also note that our method of proof utilizes the algebraic structure of number fields, and we are hopeful to extend this Voronoi type formula to arbitrary number fields. In particular, by following our ideas, one can get a Voronoi formula for the Hecke-Maass cusp forms attached to Hecke charaters of real quadratic fields with arbitrary additive twists. In our knowledge, a Voronoi formula in such generality is yet unknown.

The proof of Theorem \ref{mainthm} closely follows the method of proof for $t$-aspect subconvexity for $GL(2)$ $L$-functions of full level \cite{ka17}. We have to additionally care about the level and nebentypus.

We shall use Kloosterman's version of the circle method to detect when an integer $n$ equals $0$. Let $Q>0$ be any real number. We have,
\begin{equation}\label{delta}
\delta(n=0)=2Re\int_0^1\underset{1\leq q\leq Q<a\leq Q}{ \sum \sideset{}{^*}\sum} \frac{1}{aq}e\left(\frac{n\overline{a}}{q}-\frac{nx}{aq}\right)dx
\end{equation}
for $n\in\BZ$. Here $e(.)=e^{2\pi i.}$ and the $*$ on the inner sum means that $(a,q)=1$. $\overline{a}$ is the multiplicative inverse of $a\mod q$. The application of the circle method will not be sufficient, and we will apply a conductor lowering trick in both the $t$ and the level aspect.

By approximate functional equation,
\begin{equation*}
L(1/2+it,\psi)\ll_{D,\varepsilon}\sup_{N\ll X^{1+\varepsilon}}\frac{S(N)}{N^{1/2}} + O((tp)^{-2017})
\end{equation*}

where 
\begin{equation}\label{sn}
S(N):= \sum_{\fra\subset\calO_K} \psi(\fra) (\calN\fra)^{-it} V\left(\frac{\calN\fra}{N}\right).
\end{equation}
It therefore suffices to estimate $S(N)$. Let $\lambda(n)=\sum_{\calN\fra=n}\psi(\fra)$. We follow Munshi's approach and write (\ref{sn}) as
\begin{equation}\label{start}
S(N) = \sum_{n\geq 1}\sum_{m\geq 1}\lambda(n)m^{-it}V\left(\frac{n}{N}\right)U\left(\frac{m}{N}\right)\delta\left(\frac{n-m}{p}=0\right)\delta(n\equiv m\mod p).
\end{equation}
We use Kloosterman's delta method with Munshi's modification \cite{mu15} to detect equality. Then $S(N)=S^+(N)+S^-(N)$ where
\begin{equation}
\begin{split}
S^{\pm}(N) = \frac{1}{K}\int_0^1\int_{\BR}\frac{1}{p}\sum_{\alpha\mod p}\underset{1\leq q<Q<a\leq q+Q}{\sum \sideset{}{^*}\sum}\frac{1}{aq}\sum_{n\geq1}\sum_{m\geq1}&\lambda(n)n^{iv}m^{-i(t+v)} e\left(\pm\frac{(n-m)\bar{a}}{qp}\mp\frac{(n-m)x}{aqp}\right) \\&e\bigg(\frac{(n-m)\alpha}{p}\bigg)V\left(\frac{v}{K}\right)V\left(\frac{n}{N}\right)U\left(\frac{m}{N}\right)dv dx
\end{split}
\end{equation}
Here, $\bar{a}$ is the unique multiplicative inverse of $a(\mod q)$ inside the interval $(Q,q+Q]$ and $a$ need not be coprime with $p$. $V$ and $U$ are smooth functions supported on $[1,2]$ and $[1/2,5/2]$ respectively. $U$ is equal to $1$ on $[1,2]$. Due to the symmetry between $S^+(N)$ and $S^-(N)$, the estimates on both are exactly the same modulo some change of signs. Therefore we only consider $S^+(N)$. Separating the $n$ and $m$-sums,
\begin{equation}\label{start1}
\begin{split}
S^+(N)= \frac{1}{K}\int_0^1\int_{\BR}&\frac{1}{p}\sum_{\alpha\mod p}\underset{1\leq q<Q<a\leq q+Q}{\sum \sideset{}{^*}\sum}\frac{1}{aq}\sum_{n\geq1}\lambda(n)n^{iv}V\left(\frac{n}{N}\right) e\left(\frac{n\bar{a}}{qp} - \frac{nx}{aqp}\right)e\bigg(\frac{n\alpha}{p}\bigg)V\left(\frac{v}{K}\right) \\ & \sum_{m\geq1}m^{-i(t+v)}e\left(\frac{-m\bar{a}}{qp} + \frac{mx}{aqp}\right)e\bigg(\frac{-m\alpha}{p}\bigg) U\left(\frac{m}{N}\right) dv dx
\end{split}
\end{equation}

It will turn out that the optimal choice of $Q$ is $\sqrt{N/pK}$ (and thus lowering the conductor by $(Kp)^{1/2}$). 

We will take 
\begin{equation}
t^{3/4}p^{1/4} \ll N<t^{1+\varepsilon}p^{1/2}, p<t \textit{ and } (N/p)^{1/2}\leq K\ll N^{1-\varepsilon}
\end{equation}
In this range, we will establish the following bound.

\begin{prop}\label{mainprop}
For $p<t$ and $t^{3/4}p^{1/4} \ll N<t^{1+\epsilon}p^{1/2}$, we have
\begin{equation}
\frac{S^+(N)}{N^{1/2}}\ll t^{1/2+\varepsilon}p^{1/4}\left(\frac{K^{1/2}p^{1/4}}{N^{1/2}}+\frac{1}{K^{1/4}p^{1/4}}\right).
\end{equation}
\end{prop}
Same bound holds for $S^-(N)$, and consequently for $S(N)$. The optimal choice of $K$ is therefore $K=(N/p)^{2/3}$. With this choice of $K$, $S(N)/N^{1/2}\ll t^{1/2}p^{1/6}/N^{1/6}$. For $N\ll t^{3/4}p^{1/4}$, the trivial bound $S(N)\ll Nt^{\varepsilon}$ is sufficient. This follows by applying Cauchy's inequality to the $n$-sum followed by Lemma \ref{ramonavg} (Ramanujan bound on average). Theorem \ref{mainthm} then follows from Lemma \ref{ramonavg} and Proposition \ref{mainprop}.

Our writing style is expository and we shall justify our approach in various remarks.

\subsection{Proof Sketch of Theorem \ref{mainthm}}
We briefly explain the steps of the proof and provide heuristics in this subsection. The calculations are parallel to our previous work \cite{ka17}. The circle method is used to separate the sums on $n$ and $m$, and we arrive at (\ref{start1}). Trivial estimate gives $S(N)\ll N^{2+\varepsilon}$. For simplicity let $q\asymp Q$. We are required to save $N$ and a little more in a sum of the form
\begin{equation*}
\int_K^{2K}\frac{1}{p}\sum_{b\mod p}\sum_{q\asymp Q}\quad\sumx_{Q<a\leq Q+q}\sum_{n\asymp N}\lambda(n)n^{iv}e\left(\frac{n\overline{a}}{qp}-\frac{nx}{aqp}+\frac{nb}{p}\right)\sum_{m\asymp N}m^{-i(t+v)}e\left(\frac{-m\overline{a}}{qp}+\frac{mx}{aqp}-\frac{mb}{p}\right)dv.
\end{equation*}
The sum over $m$ has `conductor' $Qpt\asymp N^{1/2}p^{1/2}t/K^{1/2}$. Roughly speaking, the conductor takes into account the arithmetic modulus $qp$, with the size $=(t+v)$ of oscillation of the analytic weight. If we assume $K\ll t^{1-\varepsilon}$, then the size of the oscillation is $t$, so the extra oscillation of $m^{-iv}$ does not hurt us here. Poisson summation changes the length of summation to $Qpt/N$, and contributes a factor of $N$ along with a congruence condition mod $qp$ and an oscillatory integral. The oscillatory integral saves us $t^{1/2}$. In all, we will save $N/t^{1/2}$ in this step. So far the saving is independent of $K$. The next step is to apply Voronoi summation to the $n$-sum. We need to save $t^{1/2}$ in a sum of the form
\begin{equation*}
\int_K^{2K}\sum_{q\asymp Q}\underset{\begin{subarray}{c}(m,q)=1\\|m|\ll Qpt/N\end{subarray}}{\sum}\left(\frac{(t+v)aqp}{(x-ma)}\right)^{-i(t+v)}\sum_{n\asymp N}\lambda(n)e\left(\frac{nm}{qp}\right)n^{iv}e\left(-\frac{nx}{aqp}\right)dv,
\end{equation*}
where $a$ is the unique multiplicative inverse of $m\mod q$ in the range $(Q,q+Q]$. Since the $n$-sum involves $GL(2)$ Fourier coefficients, the `conductor' for the $n$-sum would be $(QpK)^2$. The new length of sum would be $(QpK)^2/N\asymp Kp$. Voronoi summation would contribute a factor of $N/qp$, a dual additive twist and an oscillatory weight function. The oscillation in the weight function would save us $K^{1/2}$. In all, we will save $Qp/K^{1/2}=N^{1/2}p^{1/2}/K$. If $K$ is large, we are actually making the bound worse. We are therefore left to save $t^{1/2}K/p^{1/2}N^{1/2}$ in $S(N)$. Using stationary phase analysis, we will be able to save $K^{1/2}$ in the integral over $v$. At this point, $K$ seems to be hurting more than helping. The final step is to get rid of the $GL(2)$ oscillations by using the Cauchy inequality and then changing the structure using Poisson summation. After Cauchy, the sum roughly looks like
\begin{equation*}
\bigg[\sum_{n\ll Kp}\bigg|\underset{\begin{subarray}{c}q\asymp Q\\ |m|\asymp Qpt/N\\(m,q)=1\end{subarray}}{\sum}e\left(-\frac{nm}{qp}\right)\int_{-K}^Kn^{-i\tau}g(q,m,\tau)d\tau\bigg|^2\bigg]^{1/2}
\end{equation*}
where $g(q,m,\tau)$ is an oscillatory weight function of size $O(1)$. The next steps would be to open the absolute value squared and, apply Poisson to the $n$-sum and analyze the $\tau$-integral. The $\tau$-integral gives us a saving of $K^{1/2}$. After Cauchy and Poisson summation, we will save $N^{1/2}/K^{1/2}$ in the diagonal term and $K^{1/4}p^{1/2}$ in the off-diagonal term. The saving over the convexity bound in the diagonal terms is $N^{1/2}/K^{1/2}p^{1/4}$. The saving over the convexity bound from the off-diagonal terms is $K^{1/4}p^{1/4}$. We will therefore get maximum saving when $N^{1/2}/K^{1/2}p^{1/4}=K^{1/4}p^{1/4}$, that is $K=(N/p)^{2/3}$. This gives us a saving of $N^{1/6}/p^{1/12}$ over the $t$- aspect convexity bound of $t^{1/2+\varepsilon}$. Matching this with the trivial bound $N^{1/2}$ for $N\ll t^{3/4}p^{1/4}$ gives us the Burgess-type bound in the $t$-aspect and a subconvex bound in the level aspect for $p<t$.

\section{$GL(2)$ Voronoi formula and Stationary Phase Method}

Let $g\in S_k^{CM}(\Gamma_0(L),\chi_{L_0})$ be a CM holomorphic cusp form of weight $k$ (an integer) and level $L$. Let $\lambda(n)$ be the Fourier coefficients of $g$. The nebentypus $\chi_{L_0}$ is induced from a character of level $L_0$ dividing $L$. For coprime integers $m$ and $q$, let $L_1=(L,q)$, $L_2=L/L_1$ and $L_3=(L_1,L_2)$ be coprime with $L_0$. Let $F$ be a smooth function compactly supported on $(0,\infty)$, and let $\tilde{F}(s)=\int_0^\infty F(x)x^{s-1}dx$ be its Mellin transform. An application of the functional equation of $L(s,f)$, followed by unwinding the integral and shifting the contour gives the Voronoi summation formula \cite{kmv}.

\begin{lem}\label{voronoi}
When $(L_1,L_2)=1$,
\begin{equation}
\begin{split}
\sum_{n\geq1}\lambda(n)e\left(n\frac{m}{q}\right)F(n)=&\frac{1}{q}\chi_{L_1}(-\overline{m})\chi_{L_2}(-q)\frac{\eta(L_2)}{\sqrt{L_2}}\sum_{n\geq1}\lambda_{L_2}(n)e\left(-n\frac{\overline{mL_2}}{q}\right)\\&\times\int_0^\infty F(x)\left[2\pi i^kJ_{k-1}\left(\frac{4\pi\sqrt{nx}}{q\sqrt{L_2}}\right)\right]dx.
\end{split}
\end{equation}
By arguments given in Appendix A, the Voronoi formula in the case $(L_1,L_2)=2$ is not much different. Let $q_1=q/2$. Then,
\begin{equation}
\begin{split}
\sum_{n\geq1}\lambda(n)e\left(n\frac{m}{q}\right)F(n)=&\frac{1}{q}\chi_{L_1}(-\overline{m})\chi_{L_2}(-q_1)\frac{\eta(L_2)}{\sqrt{L_2}}\sum_{n\geq1}\lambda_{L_2}(n)e\left(-n\frac{\overline{2mL_2}}{q_1}\right)\\&\times\int_0^\infty F(x)\left[2\pi i^kJ_{k-1}\left(\frac{4\pi\sqrt{nx}}{q\sqrt{L_2}}\right)\right]dx.
\end{split}
\end{equation}
\end{lem}

For our calculations, we take a step back and use the following representation of $J_{k-1}$ as an inverse Mellin transform,
\begin{equation}
J_{k-1}(x)=\frac{1}{2}\frac{1}{2\pi i}\int_{(\sigma)}\left(\frac{x}{2}\right)^{-s}\frac{\Gamma(s/2+(k-1)/2)}{\Gamma(1-s/2+(k-1)/2)} \quad \textit{ for }0<\sigma<1.
\end{equation}

We would also need the following bound, which is the Ramanujan conjecture on average. It follows from standard properties of Rankin-Selberg $L$-functions and is well known.
\begin{lem}\label{ramonavg}
We have,
\begin{equation*}
\sum_{n\leq x}|\lambda(n)|^2\ll_{f,\varepsilon} x^{1+\varepsilon}.
\end{equation*}
\end{lem}

We would need that stationary phase analysis done in \cite{munshi}. The contents of this section are given in Lemmas 2.3 - 2.5 of \cite{ka17}. For completeness, we mention those here.

\begin{lem}\label{stationary}
Suppose $f$ and $g$ are smooth real valued functions satisfying
\begin{equation} \label{spb}
f^{(i)}(x) \ll \Theta_f / \Omega_f^i, \quad g^{(j)}(x) \ll 1/\Omega_g^j
\end{equation}
for $i=2,3$ and $j=0,1,2$. Suppose $g(a)=g(b)=0$. Define
\begin{equation}\label{spi}
\mathcal{I} = \int_a^b g(x) e(f(x)) dx.
\end{equation}
\begin{enumerate}
\item Suppose $f'$ and $f''$ do not vanish in $[a,b]$. Let $\Lambda = \min_{[a,b]}|f'(x)|$. Then we have
\begin{equation} \label{SPB}
\mathcal{I} \ll \frac{\Theta_f}{\Omega_f^2\Lambda^3}\left(1 + \frac{\Omega_f}{\Omega_g} + \frac{\Omega_f^2}{\Omega_g^2}\frac{\Lambda}{\Theta_f/\Omega_f}\right).
\end{equation}
\item Suppose $f'$ changes sign from negative to positive at the unique point $x_0\in (a,b)$. Let $\kappa = \min\{b-x_0, x_0-a\}$. Further suppose that (\ref{spb}) holds for $i=4$ and 
\begin{equation}\label{spb1}
f^{(2)}(x) \gg \Theta_f/\Omega_f^2
\end{equation}
holds. Then
\begin{equation}\label{SP}
\mathcal{I} = \frac{g(x_0)e(f(x_0) + 1/8)}{\sqrt{f''(x_0)}} + O\left(\frac{\Omega_f^4}{\Theta_f^2\kappa^3} + \frac{\Omega_f}{\Theta_f^{3/2}} + \frac{\Omega_f^3}{\Theta_f^{3/2}\Omega_g^2}\right).
\end{equation}
\end{enumerate}
\end{lem}

We will also need a second derivative bound for integrals in two variables. Let
\begin{equation}
\mathcal{I}_{(2)}=\int_a^b\int_c^d g(x,y)e(f(x,y))dydx.
\end{equation}
with $f$ and $g$ smooth real valued functions. Let supp$(g)\subset(a,b)\times(c,d)$. Let $r_1, r_2$ be such that inside the support of the integral,
\begin{equation}\label{2nd.der.bound}
f^{(2,0)}(x,y)\gg r_1^2, \quad f^{(0,2)}(x,y)\gg r_2^2, \quad f^{(2,0)}(x,y)f^{(0,2)}(x,y)-\left[f^{(1,1)}(x,y)\right]^2\gg r_1^2r_2^2,
\end{equation}
where $f^{(i,j)}(x,y)=\frac{\partial^{i+j}}{\partial x^i\partial y^j}f(x,y)$. Then we have (see \cite{srini65}),
\begin{equation*}
\mathcal{I}_{(2)}\ll\frac{1}{r_1r_2}.
\end{equation*}
Define the total variance of $g$ by
\begin{equation*}
\textit{var}(g):=\int_a^b\int_c^d\left|\frac{\partial^2}{\partial x\partial y}g(x,y)\right|dydx.
\end{equation*}
Integration by parts along with the above bound gives us the following.
\begin{lem}\label{2nd.der.lem}
Suppose $f, g, r_1, r_2$ are as above and satisfy condition (\ref{2nd.der.bound}). Then we have
\begin{equation*}
\mathcal{I}_{(2)}\ll\frac{\textit{var}(g)}{r_1r_2}.
\end{equation*}
\end{lem}

\subsection{An integral of interest}
Following Munshi \cite{munshi}, let $W$ be a smooth real valued function with supp$(W)\subset[a,b]\subset(0,\infty)$ and $W^{(j)}(x)\ll_{a,b,j}1$. Define
\begin{equation}
W^{\dagger}(r,s)\int_0^\infty W(x)e(-rx)x^{s-1}dx
\end{equation}
where $r\in\BR$ and $s=\sigma+i\beta\in\BC$. This integral is of the form (\ref{spi}) with
\begin{equation*}
g(x)=W(x)x^{\sigma-1} \quad \textit{ and } \quad f(x)=-rx+\frac{1}{2\pi}\beta\log x.
\end{equation*}
Then,
\begin{equation*}
f'(x)=-r+\frac{1}{2\pi}\frac{\beta}{x}\quad\textit{ and }\quad f^{(j)}(x)=(-1)^j(j-1)!\frac{1}{2\pi}\frac{\beta}{x^j}
\end{equation*}
for $j\geq2$. The unique stationary phase occurs at $x_0=\beta/2\pi r$. Note that we can write
\begin{equation}
f'(x)=\frac{\beta}{2\pi}\left(\frac{1}{x}-\frac{1}{x_0}\right)=r\left(\frac{x_0}{x}-1\right).
\end{equation}
Applying Lemma \ref{stationary} appropriately to $W^{\dagger}(r,s)$, we get the following.
\begin{lem}\label{dagger}
Let $W$ be a smooth real valued function with supp$(W)\subset[a,b]\subset(0,\infty)$ and $W^{(j)}(x)\ll_{a,b,j}1$. Let $r\in\BR$ and $s=\sigma+i\beta\in\BC$. We have
\begin{equation}
W^{\dagger}(r,s)=\frac{\sqrt{2\pi}e(1/8)}{\sqrt{-\beta}}W\left(\frac{\beta}{2\pi r}\right)\left(\frac{\beta}{2\pi r}\right)^\sigma\left(\frac{\beta}{2\pi er}\right)^{i\beta}+O_{a,b,\sigma}\left(\min\{|\beta|^{-3/2},|r|^{-3/2}\}\right).
\end{equation}
We also have
\begin{equation}
W^{\dagger}(r,s)=O_{a,b,j,\sigma}\left(\min\left\lbrace\left(\frac{1+|\beta|}{|r|}\right)^j,\left(\frac{1+|r|}{|\beta|}\right)^j\right\rbrace\right).
\end{equation}
\end{lem}
 
\section{Application of Dual summation formulas}

\subsection{Poisson summation to the $m$-sum}
The $m$-sum is given by
\begin{equation*}
\sum_{m\geq1}m^{-i(t+v)}e\left(\frac{-m\bar{a}-m\alpha q}{qp} + \frac{mx}{aqp} \right)U\left(\frac{m}{N}\right).
\end{equation*}
Breaking the $m$-sum into congruence classes modulo $qp$ by changing variables $m\mapsto\beta+mqp$, we get
\begin{equation*}
\sum_{\beta\mod qp}e\left(\frac{-\beta\overline{a}-\beta\alpha q}{qp}\right)\sum_{m\in\BZ}(\beta+mqp)^{-i(t+v)}e\left(\frac{(\beta+mqp)x}{aqp}\right)U\left(\frac{\beta+mqp}{N}\right).
\end{equation*}
Poisson summation to the $m$-sum gives
\begin{equation*}
\sum_{\beta\mod qp}e\left(\frac{-\beta\overline{a}-\beta\alpha q}{qp}\right)\sum_{m\in\BZ}\int_{\BR}(\beta+yqp)^{-i(t+v)}e\left(\frac{(\beta+yqp)x}{aqp}\right)U\left(\frac{\beta+yqp}{N}\right)e(-my)dy.
\end{equation*}
Letting $u=(\beta+yqp)/N$ and executing the complete character sum $\mod qp$, we arrive at
\begin{equation}\label{msum}
N^{1-i(t+v)}\underset{\begin{subarray}{c} m\in\BZ \\ m\equiv\bar{a}+\alpha q(\mod qp)\end{subarray}}{\sum}\int_{\BR} u^{-i(t+v)}e\left(\frac{N(x-ma)}{aqp}u\right) U(u) du 
\end{equation}
The above integral equals
\begin{equation}
U^{\dagger}\left(\frac{N(ma-x)}{aqp}, 1-i(t+v)\right).
\end{equation}
Since $m\equiv\bar{a}+\alpha q\bmod qp$, we have $m\equiv\overline{a}\bmod q$ and $\alpha\equiv\frac{m-\bar{a}}{q}\bmod p$. Therefore $\alpha$ is determined $
\mod p$, and the $\alpha$-sum in the expression of $S^+(N)$ vanishes. Everything together,
\begin{equation}
\begin{split}
S^+(N) = &\frac{1}{K} \int_0^1\int_{\R} V\left(\frac{v}{K}\right) \underset{1\leq q\leq Q<a\leq Q}{ \sum \sideset{}{^*}\sum} \frac{1}{aqp} \sum_{n\geq 1} \lambda(n) n^{iv}V\left(\frac{n}{N}\right)\\ &\times N^{1-i(t+v)} \sum_{m\equiv\bar{a}\mod q} e\left(\frac{nm}{qp} - \frac{nx}{aqp} \right) U^\dagger\left(\frac{N(ma-x)}{aqp}, 1-i(t+v)\right) dv dx.
\end{split}
\end{equation}
We first observe that $m=0$ can occur only when $q=1$ and bounds on $U^{\dagger}$ from lemma \ref{dagger} give arbitrary saving as soon as $Q$ has a size in $N$, i.e. as soon as $Q=N^{\varepsilon}$ for any $\varepsilon>0$. With the assumption $K\ll t$, lemma \ref{dagger} gives arbitrary saving for $|m|\gg qpt^{1+\epsilon}/N$. The sum becomes
\begin{equation}
\begin{split}
S^+(N)= \frac{N^{1-it}}{K}\int_0^1\int_{\BR}&N^{-iv}\underset{1\leq q<Q}{\sum}\underset{\begin{subarray}{c} 1\leq |m|\ll qpt^{1+\epsilon}/N \\ (m,q)=1 \end{subarray}}{\sum} \frac{1}{aqp} U^{\dagger}\left(\frac{N(ma-x)}{aqp}, 1-i(t+v)\right)V\left(\frac{v}{K}\right) \\& \sum_{n\geq1}\lambda(n) n^{iv}V\left(\frac{n}{N}\right)e\left(\frac{nm}{qp}-\frac{nx}{aqp}\right)dvdx
\end{split}
\end{equation}

We next split the $q-$sum into dyadic segments $(C,2C]$ 
\begin{equation*}
S^+(N) = \frac{N}{K} \underset{\begin{subarray}{c}1\leq C\leq Q\\ \textit{ dyadic }\end{subarray}}{\sum} S(N,C)
\end{equation*}
where
\begin{equation}
\begin{split}
S(N,C) = \int_0^1\int_{\R} N^{-i(t+v)}V\left(\frac{v}{K}\right)&\sum_{C< q\leq 2C}\underset{(m,q)=1}{\sum_{1\leq |m|\ll \frac{qpt^{1+\epsilon}}{N}}}\frac{1}{aqp} U^\dagger\left(\frac{N(ma-x)}{aqp},1-i(t+v)\right) \\ & \sum_{n\geq1} \lambda(n) n^{iv}V\left(\frac{n}{N}\right) e\left(\frac{nm}{qp}-\frac{nx}{aqp}\right) dv dx.
\end{split}
\end{equation}

\subsection{Voronoi summation to the $n$-sum} 
Let $(m,p)=p^\delta$, for $\delta=0, 1$. When $-D\equiv2,3\bmod4$, the level of the cusp form is $L=4Dp$. When $-D\equiv1\bmod4$, the level is $L=Dp$. In the former case, $L_1=(4D,q)p^{1-\delta}$ and $L_2=4Dp^{\delta}/(4D,q)$. In the latter case, $L_1=(D,q)p^{1-\delta}$ and $L_2=D_qp^{\delta}$. In the latter case, $(L_1,L_2)=1$ and we can use the Voronoi summation formula in \cite{kmv}. However in the former case, $(L_1,L_2)=1, 2$ depending on the congruence class of $D$ and $q$ $\bmod2$. We therefore take up the case where \cite{kmv} is insufficient. The calculations for $-D\equiv2,3\bmod4$ are similar, that's why we present the details of the case $-D\equiv3\bmod4$.

For ease of notation, let $g=(D,q), D_q=D/(D,q)$ and $4_q=4/(4,q)$. Then $L_1=(4,q)gp^{1-\delta}$ and $L_2=4_qD_qp^\delta$. $L_3=2$ if $q\equiv2\bmod4$ and $L_3=1$ otherwise. Due to lemma \ref{voronoi}, the calculations for the cases $q\equiv2\bmod4$ and $q\nequiv2\bmod4$ are similar, so it suffices to work with one of the cases. When $q\nequiv2\bmod4$,
\begin{equation*}
\begin{split}
\sum_{n\geq1}\lambda(n)e\left(\frac{nm}{qp}\right)n^{iv}e\left(\frac{-nx}{aqp}\right)V\left(\frac{n}{N}\right)=&\frac{2\pi i^kN^{1+iv}}{4\pi iqp^{1-\delta}}\chi_{L_1}(-\overline{m}p^{\delta})\chi_{L_2}(-qp^{1-\delta})\frac{\eta(4_qD_qp^\delta)}{\sqrt{4_qD_qp^\delta}}\sum_{n\geq1}\lambda_{L_2}(n)e\left(-n\frac{\overline{m4_qD_q}}{qp^{1-\delta}}\right)\\&\times\int_0^\infty y^{iv}e\left(\frac{-yNx}{aqp}\right)V(y)\int_{(\sigma)}\left(\frac{2\pi\sqrt{Nny}}{qp^{1-\delta/2}\sqrt{4_qD_q}}\right)^{-s}\gamma_k(s)dsdy. 
\end{split}
\end{equation*}
where 
\begin{equation}
\gamma_k(s) = (2\pi)^{-s}\frac{\Gamma(s/2 + (k-1)/2)}{\Gamma(1-s/2+(k-1)/2)}.
\end{equation}
Since $k>1$, $\gamma_k$ has no poles in the region $Re(s)>-1$, and we can shift the $s$-integral to $\sigma=-1/2$ without picking up poles. This allows us to interchange the $s$ and $x$ integrals.
\begin{equation}
\begin{split}
\sum_{n\geq1}\lambda(n)e\left(\frac{nm}{qp}\right)n^{iv}e\left(\frac{-nx}{aqp}\right)V\left(\frac{n}{N}\right)=&\frac{2\pi i^kN^{1+iv}}{4\pi iqp^{1-\delta}}\chi_{L_1}(-\overline{m}p^{\delta})\chi_{L_2}(-qp^{1-\delta})\frac{\eta(L_2)}{\sqrt{4_qD_qp^\delta}}\sum_{n\geq1}\lambda_{L_2}(n)e\left(-n\frac{\overline{m4_qD_q}}{qp^{1-\delta}}\right)\\&\times\int_{(\sigma)}\left(\frac{2\pi\sqrt{Nn}}{qp^{1-\delta/2}\sqrt{4_qD_q}}\right)^{-s}\gamma_k(s)V^\dagger\left(\frac{Nx}{aqp},1-s/2+iv\right)ds. 
\end{split}
\end{equation}
The bound on $V^\dagger$ gives
\begin{equation}
V^\dagger\left(\frac{xN}{aqp}, 1-s/2+iv \right) \ll_j \min \left\lbrace 1 , \left(\frac{1 + |Nx/aqp|}{|v-\tau/2|}\right)^j \right\rbrace.
\end{equation}
We can therefore shift the $s$-integral to $\sigma=M$ for large $M$ and get arbitrary saving for large $n$. With arguments similar to Remark 3.2 of \cite{ka17}, we will get arbitrary saving for $n>\max\left\lbrace Q^2p^{2-\delta}K^2D/N, DN/Q^2p^\delta\right\rbrace t^\varepsilon$. So the optimal choice of $Q=\sqrt{N/pK}$. Thus the introduction of divisibility by $p$ in (\ref{start}) is a conductor lowering trick. For smaller values of $n$, we shift the integral to $\sigma=1$. Note that $\gamma_k(1+i\tau)=O(1)$.

Assuming $K\ll t^{1-\epsilon}$, we get arbitrary saving for $|\tau| > Nt^\epsilon/QCp$ due to the bounds on $V^\dagger$. Thus we can restrict the integral to $\tau \in [-Nt^\epsilon/QCp, Nt^\epsilon/QCp]$ by defining a smooth partition of unity on this set. Let $W_J$  for $J\in\mathcal{J}$ be smooth bump functions satisfying $x^lW_J^{(l)}\ll_l 1$ for all $l\geq0$. For $J=0$, let the support of $W_0$ be in $[-1,1]$ and for $J>0$ (resp. $J<0$), let the support of $W_J$ be in $[J,4J/3]$ (resp $[4J/3,J]$). Finally, we require that
\begin{equation*}
\sum_{J\in \mathcal{J}} W_J(x) = 1 \text{\quad for \quad $x \in [-Nt^\epsilon/QCp, Nt^\epsilon/QCp]$}
\end{equation*}
The precise definition of the functions $W_J$ will not be needed. We note that we need only $O(\log(t))$ such $J\in\mathcal{J}$.
Then,
\begin{equation}\label{snc}
\begin{split}
S(N,C) =&\sum_{\delta\in\{0,1\}}\underset{\begin{subarray}{c}L_1|L\\ p^{1-\delta}|L_1\end{subarray}}{\sum}\frac{i^kN^{1/2-it}K}{2}\sum_{J\in\mathcal{J}}\sum_{n\ll KDp^{1-\delta}t^\varepsilon} \frac{\lambda_{L_2}(n)}{n^{1/2}}\underset{\begin{subarray}{c}C<q\leq2C\\L_1=(4D,q)p^{1-\delta}\end{subarray}}{\sum}\underset{\begin{subarray}{c}(m,q)=1\\(m,p)=p^\delta\\1\leq|m|\ll \frac{qpt^{1+\epsilon}}{N}\end{subarray}}{\sum}\chi_{L_1}(-\overline{m}p^{\delta})\chi_{L_2}(-qp^{1-\delta})\\&\times\eta(L_2)\frac{1}{aqp}e\left(-n\frac{\overline{m4_qD_q}}{qp^{1-\delta}}\right)\int_{\R} \left(\frac{2\pi\sqrt{Nn}}{qp^{1-\delta/2}\sqrt{4_qD_q}}\right)^{-i\tau}\gamma\left(1 + i\tau\right) W_J(\tau) \mathcal{I^{**}}(q,m,\tau) d\tau + O(t^{-2015}),
\end{split}
\end{equation}
where
\begin{equation}\label{I**}
\mathcal{I^{**}}(q,m,\tau)= \int_0^1 \int_{\R} V(v) U^{\dagger}\left( \frac{N(ma-x)}{aqp}, 1-i(t+Kv) \right) V^{\dagger}\left( \frac{Nx}{aqp}, \frac{1}{2} - \frac{i\tau}{2} + iKv \right) dv dx.
\end{equation}

\begin{rem}\label{rem1}
Trivially bounding the sums shows
\begin{equation*}
S(N,C)\ll \frac{K^{5/2}p^{1/2}t^{1+\varepsilon}}{N^{1/2}},
\end{equation*}
which is worse than the convexity bound by a factor of $t^{1/2}p^{1/4}K^{3/2}$. The next step would be to use stationary phase analysis on the $v$-integral to get some further saving.
\end{rem}

We shall analyze the integral $I^{**}$ in eqn (\ref{I**}) using stationary phase analysis and write $I^{**}(q,m,\tau)=I_1(q,m,\tau)+I_2(q,m,\tau)$. $I_1$ comes from the area on the $x$-$v$ plane where a stationary phase exists, while $I_2$ comes from the rest of the area which gives a smaller contribution.

\subsection{Stationary Phase Analysis on $I^{**}(q,m,\tau)$}

Stationary Phase Analysis on $U^\dagger$ and $V^\dagger$ gives,

\begin{equation}
\begin{split}
U^\dagger\left(\frac{N(ma-x)}{aqp},1-i(t+Kv)\right)=&\frac{-\sqrt{2\pi}e(1/8)}{\sqrt{t+Kv}}U\left(\frac{(t+Kv)aqp}{2\pi N(x-ma)}\right)\left(\frac{(t+Kv)aqp}{2\pi N(x-ma)}\right)\\&\left(\frac{(t+Kv)aqp}{2\pi eN(x-ma)}\right)^{-i(t+Kv)}+O(t^{-3/2})
\end{split}
\end{equation}

\begin{equation}
\begin{split}
V^\dagger\left(\frac{Nx}{aqp},\frac{1}{2}-\frac{i\tau}{2}-iKv\right)=&\frac{-\sqrt{2\pi}e(1/8)}{\sqrt{\tau/2-Kv}}V\left(\frac{(Kv-\tau/2)aqp}{2\pi Nx}\right)\left(\frac{(Kv-\tau/2)aqp}{2\pi Nx}\right)^{1/2}\\&\left(\frac{(Kv-\tau/2)aqp}{2\pi eNx}\right)^{i(Kv-\tau/2)}+O\left(\min\left\lbrace|\tau/2-Kv|^{-3/2},\left|\frac{aqp}{Nx}\right|^{3/2}\right\rbrace\right)
\end{split}
\end{equation}

Multiplying these together,

\begin{equation}\label{analyzedI**}
I^{**}(q,m,\tau)=\int\int\textit{Main terms}+O\left(t^{-1/2}\int_0^1\int_1^2\min\left\lbrace|\tau/2-Kv|^{-3/2},\left|\frac{aqp}{Nx}\right|^{3/2}\right\rbrace dvdx\right)
\end{equation}

\begin{rem}
Analysis as done in \cite{ka17} shows that the error term is $O\left(t^{-1/2+\varepsilon}K^{-3/2}\min\left\lbrace1,\frac{10K}{|\tau|}\right\rbrace\right)$. That cancels extra $N^{1/2}K^{3/2}$ over the convexity bound as mentioned in Remark (\ref{rem1}). If we now bound the $\tau$-integral in eqn (\ref{snc}), the $\tau$-integral would be
\begin{equation*}
\ll \int_{-Nt^{\varepsilon}/QpC}^{Nt^{\varepsilon}/QpC}\min\left\lbrace1,\frac{10K}{|\tau|}\right\rbrace d\tau \ll Kt^{\varepsilon}.
\end{equation*}
This again lands us right at the convexity bound when $K$ has no size. Therefore to get some further saving and beat the convexity bound, we will apply Cauchy inequality followed by Poisson summation to the $n$-sum.
\end{rem}

\subsubsection{Main terms}
The integral over the main terms in eqn (\ref{analyzedI**}) is,
\begin{equation}\label{mainterms}
\begin{split}
\left(\frac{2\pi aqp}{Nt}\right)^{1/2}\int_0^1\frac{1}{x^{1/2}}\int_1^2&\frac{t^{1/2}(t+Kv)^{1/2}aqp}{2\pi N(x-ma)}U\left(\frac{(t+Kv)aqp}{2\pi N(x-ma)}\right)V\left(\frac{(Kv-\tau/2)aqp}{2\pi Nx}\right)V(v)\\&\left(\frac{(t+Kv)aqp}{2\pi eN(x-ma)}\right)^{-i(t+Kv)}\left(\frac{(Kv-\tau/2)aqp}{2\pi eNx}\right)^{i(Kv-\tau/2)}dvdx
\end{split}
\end{equation}

\begin{rem}
Trivial estimate gives $I^{**}(q,m,\tau)\ll\frac{1}{(tK)^{1/2}}$. We need to save $K$ and a bit more to beat the convexity bound. Note that without this $K$, we would be stuck at the convexity bound. Although $K$ seems to be hurting us rather than helping right now, the final saving will come in the last step of Cauchy and Poisson to the $(c,f)$-sum.
\end{rem}

\begin{rem}
Due to the weight function, $U\left(\frac{(t+Kv)aqp}{2\pi N(x-ma)}\right)$, $m$ is truly of size $tqp/N$ (and $m<0$) instead of the full range $1\leq|m|\ll t^{1+\varepsilon}qp/N$, thus saving $t^{\varepsilon}$.
\end{rem}

We'll now do stationary phase analysis on the $v$-integral. This would save us another $K^{1/2}$. Writing the integral over $v$ in eqn (\ref{mainterms}) as $\int_1^2g(v)e(f(v))dv$, we have
\begin{equation}
f(v)=\frac{-(t+Kv)}{2\pi}\log\left(\frac{(t+Kv)aqp}{2\pi eN(x-ma)}\right)+\frac{Kv-\tau/2}{2\pi}\log\left(\frac{(Kv-\tau/2)aqp}{2\pi eNx}\right)
\end{equation}
and
\begin{equation}\label{g(v)}
g(v)=\frac{t^{1/2}(t+Kv)^{1/2}aqp}{2\pi N(x-ma)}U\left(\frac{(t+Kv)aqp}{2\pi N(x-ma)}\right)V\left(\frac{(Kv-\tau/2)aqp}{2\pi Nx}\right)V(v)
\end{equation}

The stationary phase occurs at,
\begin{equation}
v_0=\frac{-(\tau/2+t)x}{Kma}+\frac{\tau}{2K}.
\end{equation}
Also, for $j\geq2$
\begin{equation}
f^{(j)}(v)=\frac{K}{2\pi}(-1)^j(j-2)!\left(\frac{K^{j-1}}{(Kv-\tau/2)^{j-1}}-\frac{K^{j-1}}{(t+Kv)^{j-1}}\right).
\end{equation}
In the support of the integral, $1/(Kv-\tau/2)\asymp aqp/Nx\geq aqp/N\gg K/t\geq K/(t+Kv)$. Therefore, in the support of the integral
\begin{equation*}
f^{(j)}(v)\asymp_j K\left(\frac{aqpK}{Nx}\right)^{j-1} \quad \textit{ for }j\geq2.
\end{equation*}
Also, in the support of the integral,
\begin{equation*}
g^{(j)}(v)\ll_j \left(1+\frac{Kaqp}{Nx}\right)^j \quad \textit{ for }j\geq0.
\end{equation*}

We can write,
\begin{equation}
f'(v)=\frac{K}{2\pi}\log\left(1+\frac{K(v_0-v)}{(t+Kv)}\right)-\frac{K}{2\pi}\log\left(1+\frac{K(v_0-v)}{(Kv-\tau/2)}\right).
\end{equation}
Due to the weight function $V(v)$ in expression (\ref{g(v)}), there is no stationary phase if $v_0\notin[0.5,2.5]$. In that case $|v_0-v|\geq0.5$ and
\begin{equation*}
|f'(v)|\gg K\min\left\lbrace1,\frac{Kaqp}{Nx}\right\rbrace
\end{equation*}
which is obtained by Taylor expanding the logarithms in the above expression. We'll use the following result due to Huxley.

\begin{lem}\label{stationaryphase}
Consider the integral,
\begin{equation*}
\mathfrak{I}=\int_a^b g(v)e(f(v))dv
\end{equation*}
where $g$ is supported on $[a,b]\subset(0,\infty)$. Let $\Theta_f, \Omega_f$ and $\Omega_g$ be such that $\Theta_f, \Omega_f\gg(b-a)$, and
\begin{equation}
f^{(i)}(v)\ll\Theta_f/\Omega_f^{i}, \quad g^{(j)}(v)\ll 1/\Omega_g^j.
\end{equation}
\begin{enumerate}
\item[(1)] If $f'(v)$ does not vanish on $[a,b]$, let $\Lambda=\min_{[a,b]}|f'(v)|$. Then,
\begin{equation}
\mathfrak{I}\ll \frac{\Theta_f}{\Omega_f^2\Lambda^3}\left(1+\frac{\Omega_f}{\Omega_g}+\frac{\Omega_f^2}{\Omega_g^2}\frac{\Lambda}{\Theta_f/\Omega_f}\right).
\end{equation}
\item[(2)] If $f'(v)$ vanishes at $v_0\in[a,b]$, let $\kappa=\min\{b-v_0,v_0-a\}$ and assume $f^{(2)}(v)\gg \Theta_f/\Omega_f^2$. Then,
\begin{equation}
\mathfrak{I}=\frac{g(v_0)e(f(v_0)+1/8)}{\sqrt{f''(v_0)}}+O\left(\frac{\Omega_f^4}{\Theta_f^3\kappa}+\frac{\Omega_f}{\Theta_f^{3/2}}+\frac{\Omega_f^3}{\Theta_f^{3/2}\Omega_g^2}\right).
\end{equation}
\end{enumerate}
\end{lem}


Trivially bounding the integral in expression (\ref{mainterms}) for the range $x\in[0,1/K]$,
\begin{equation}\label{trivialbound}
\begin{split}
\ll&\left(\frac{aqp}{Nt}\right)^{1/2}\int_0^{1/K}\frac{1}{x^{1/2}}\int_1^2\frac{t^{1/2}(t+Kv)^{1/2}aqp}{2\pi N(x-ma)}U\left(\frac{(t+Kv)aqp}{2\pi N(x-ma)}\right)V\left(\frac{(Kv-\tau/2)aqp}{2\pi Nx}\right)V(v)dvdx\\
\ll&\frac{1}{t^{1/2}}\left(\frac{N}{aqp}\right)^{1/2}\frac{1}{K^{5/2}}.
\end{split}
\end{equation}
For the range $x\in[1/K,1]$, we use the bound in lemma (\ref{stationaryphase}). In the case there is no stationary phase, we will use the first statement of lemma (\ref{stationaryphase}). We have,
\begin{equation}
\Theta_f=\frac{Nx}{aqp},\quad\Omega_f=\frac{Nx}{aqpK},\quad\Lambda=K\min\left\lbrace1,\frac{Kaqp}{Nx}\right\rbrace,\quad\Omega_g=\min\left\lbrace1,\frac{Nx}{aqpK}\right\rbrace.
\end{equation}

Next is the contribution of $x\in[1/K,1]$ when there is no stationary phase. When $x<aqDK/N$, $\Lambda=K$ and $\Omega_g=\Omega_f$. In that case, the contribution is 
\begin{equation*}
\left(\frac{2\pi aqp}{Nt}\right)^{1/2}\int_{1/K}^{\max\{\frac{1}{K},\frac{Kaqp}{N}\}}\frac{1}{x^{1/2}}\frac{aqp}{NKx}dx\ll \frac{1}{t^{1/2}K^{2}}.
\end{equation*}
This is always smaller than the contribution of (\ref{trivialbound}). When $x>aqpK/N$, $\Lambda=K^2aqp/Nx$ and $\Omega_g=1$. In that case, the contribution is $1/K^3t^{1/2}$, which is better than above. We next calculate the contribution of the error term when there is a stationary phase. For that we have $\kappa>0.4$. One can calculate that for both $x<aqpK/N$ and $x>aqpK/N$, the contribution is $1/K^2t^{1/2}$. 

we summarize the analysis in the following Lemma. Let
\begin{equation}
B(C,\tau)= \frac{t^\varepsilon}{t^{1/2}K^{3/2}}\min\left\lbrace1,\frac{10K}{|\tau|}\right\rbrace + \frac{1}{t^{1/2}K^{5/2}}\left(\frac{N}{QCp}\right)^{1/2}.
\end{equation}

Note that,
\begin{equation}\label{bctau}
\int_{-Nt^{\varepsilon}/QCp}^{Nt^{\varepsilon}/QCp}B(C,\tau)d\tau\ll \frac{K}{t^{1/2}K^{3/2}}+\frac{1}{t^{1/2}K^{5/2}}\left(\frac{N}{QCp}\right)^{3/2}.
\end{equation}

\begin{lem}\label{lemmaI**}
Suppose $C<q\leq 2C$, with $1\ll C\leq (N/Kp)^{1/2}$ and $K$ satisfies $1\leq K\ll t^{1-\epsilon}$. Suppose $t>2$ and $|\tau|\ll N^{1/2}K^{1/2}t^{\epsilon}/p^{1/2}$. We have
\begin{equation*}
\mathcal{I}^{**}(q,m,\tau) = \mathcal{I}_1(q,m,\tau) + \mathcal{I}_2(q,m,\tau)
\end{equation*}
where
\begin{equation*}
\mathcal{I}_1(q,m,\tau) = \frac{c_4}{(t+\tau/2)^{1/2}K}\left(-\frac{(t+\tau/2)qp}{2\pi eNm}\right)^{3/2-i(t+\tau/2)}V\left(-\frac{(t+\tau/2)qp}{2\pi Nm}\right)\int_0^1 V\left(\frac{\tau}{2K} - \frac{(t+\tau/2)x}{Kma}\right) dx
\end{equation*}
for some absolute constant $c_4$ and
\begin{equation*}
\mathcal{I}_2(q,m,\tau) := \mathcal{I}^{**}(q,m,\tau) -  \mathcal{I}_1(q,m,\tau) = O(B(C,\tau)t^{\epsilon})
\end{equation*}
with $B(C,\tau)$ as defined in (\ref{bctau}).
\end{lem}
Consequently, we have the following decomposition of $S(N,C)$.

\begin{lem}
\begin{equation*}
S(N,C) = \sum_{J\in\mathcal{J}}\{S_{1, J}(N,C) + S_{2,J}(N, C)\} + O(t^{-2015})
\end{equation*}
where
\begin{equation*}
\begin{split}
S_{l,J}(N,C) =&\sum_{\delta\in\{0,1\}}\underset{\begin{subarray}{c}L_1|L\\ p^{1-\delta}|L_1\end{subarray}}{\sum}\frac{i^kN^{1/2-it}K}{2}\sum_{n\ll KDp^{1-\delta}t^\varepsilon} \frac{\lambda_{L_2}(n)}{n^{1/2}}\underset{\begin{subarray}{c}C<q\leq2C\\L_1=(4D,q)p^{1-\delta}\end{subarray}}{\sum}\underset{\begin{subarray}{c}(m,q)=1\\(m,p)=p^\delta\\1\leq|m|\ll \frac{qpt^{1+\epsilon}}{N}\end{subarray}}{\sum}\chi_{L_1}(-\overline{m}p^{\delta})\chi_{L_2}(-qp^{1-\delta})\\&\times\eta(L_2)\frac{1}{aqp}e\left(-n\frac{\overline{m4_qD_q}}{qp^{1-\delta}}\right)\mathcal{I}_{l,J}(q,m,n) + O(t^{-2015})
\end{split}
\end{equation*}
and
\begin{equation*}
\mathcal{I}_{l,J}(q,m,n) = \int_{\R} \left(\frac{2\pi\sqrt{Nn}}{qp^{1-\delta/2}\sqrt{4_qD_q}}\right)^{-i\tau}\gamma\left(1 + i\tau\right) W_J(\tau) \mathcal{I}_{l,J}(q,m,\tau) d\tau,
\end{equation*}
with $\mathcal{I}_{l,J}(q,m,\tau)$ defined by having an extra factor of $W_J(\tau)$ in the $\tau$-integral of $\mathcal{I}_l(q,m,\tau)$.
\end{lem}

\section{Application of Cauchy inequality and Poisson summation- I}

In this section, we will estimate
\begin{equation*}
S_2(N,C) := \sum_{J\in\mathcal{J}} S_{2,J}(N,C)
\end{equation*}
We'll not utilize any cancellation over the $\tau$-integral. Dividing the $n$-sum into dyadic segments and using the bound $\gamma(1+i\tau)\ll 1$,  we get
\begin{equation}
\begin{split}
S_2(N,C) \ll &\sum_{\delta\in\{0,1\}}\underset{\begin{subarray}{c}L_1|L\\ p^{1-\delta}|L_1\end{subarray}}{\sum} t^{\epsilon}N^{1/2}K\int_{-\frac{(NK)^{1/2}t^{\epsilon}}{p^{1/2}C}}^{\frac{(NK)^{1/2}t^{\epsilon}}{p^{1/2}C}}\underset{dyadic}{\sum_{1\leq R \ll KDp^{1-\delta}t^{\epsilon}}}\sum_n \frac{|\lambda(n)|}{n^{1/2}}U\left(\frac{n}{R}\right)\\&\times\bigg|\underset{\begin{subarray}{c}C<q\leq2C\\L_1=(4D,q)p^{1-\delta}\end{subarray}}{\sum}\underset{\begin{subarray}{c}(m,q)=1\\(m,p)=p^\delta\\1\leq|m|\ll \frac{qpt^{1+\epsilon}}{N}\end{subarray}}{\sum}\chi_{L_1}(-\overline{m}p^{\delta})\chi_{L_2}(-qp^{1-\delta})\frac{1}{a(qp)^{1-i\tau}p^{i\tau\delta/2}}e\left(-n\frac{\overline{m4_qD_q}}{qp^{1-\delta}}\right)\mathcal{I}_2(q,m,\tau)\bigg|d\tau
\end{split}
\end{equation}

We apply Cauchy inequality to the $n$-sum and write,
\begin{equation}
S_2(N,C)\ll \sum_{\delta\in\{0,1\}}\underset{\begin{subarray}{c}L_1|L\\ p^{1-\delta}|L_1\end{subarray}}{\sum}t^\varepsilon N^{1/2}K\int_{-\frac{Nt^\varepsilon}{CQp}}^{\frac{Nt^\varepsilon}{CQp}}\underset{dyadic}{\sum_{1\leq R \ll KDp^{1-\delta}t^{\epsilon}}} R^{1/2}[S_{2,\delta,L_1}(N,C,R,\tau)]^{1/2}d\tau,
\end{equation}
where
\begin{equation}
S_{2,\delta,L_1}(N,C,R,\tau)=\sum_{n}\frac{1}{n}U\left(\frac{n}{R}\right)\bigg|\underset{\begin{subarray}{c}C<q\leq2C\\L_1=(4D,q)p^{1-\delta}\end{subarray}}{\sum}\underset{\begin{subarray}{c}(m,q)=1\\(m,p)=p^\delta\\1\leq|m|\ll \frac{qpt^{1+\epsilon}}{N}\end{subarray}}{\sum}\frac{\chi_{L_1}(-\overline{m}p^{\delta})\chi_{L_2}(-qp^{1-\delta})}{a(qp)^{1-i\tau}p^{i\tau\delta/2}}e\left(-n\frac{\overline{m4_qD_q}}{qp^{1-\delta}}\right)\mathcal{I}_2(q,m,\tau)\bigg|^2.
\end{equation}

Opening the absolute value squared,
\begin{equation}
\begin{split}
S_{2,\delta,L_1}(N,C,R,\tau)=&\underset{\begin{subarray}{c}C<q\leq2C\\L_1=(4D,q)p^{1-\delta}\end{subarray}}{\sum}\underset{\begin{subarray}{c}(m,q)=1\\(m,p)=p^\delta\\1\leq|m|\ll \frac{qpt^{1+\epsilon}}{N}\end{subarray}}{\sum}\underset{\begin{subarray}{c}C<q'\leq2C\\L_1=(4D,q')p^{1-\delta}\end{subarray}}{\sum}\underset{\begin{subarray}{c}(m',q')=1\\(m',p)=p^\delta\\1\leq|m'|\ll \frac{q'pt^{1+\epsilon}}{N}\end{subarray}}{\sum} \frac{\chi_{L_1}(-\overline{m}p^{\delta})\chi_{L_2}(-qp^{1-\delta})}{a(qp)^{1-i\tau}p^{i\tau\delta/2}}\\&
\times \frac{\overline{\chi}_{L_1}(-\overline{m'}p^{\delta})\overline{\chi}_{L_2}(-q'p^{1-\delta})}{a'(q'p)^{1+i\tau}p^{-i\tau\delta/2}}I_2(q,m,\tau)\overline{I_2(q',m',\tau)}\times\mathbf{T}
\end{split}
\end{equation}
where
\begin{equation}
\mathbf{T}=\sum_n\frac{1}{n}U\left(\frac{n}{R}\right)e\left(-n\frac{\overline{m4_qD_q}}{qp^{1-\delta}}\right)e\left(n\frac{\overline{m'4_{q'}D_{q'}}}{q'p^{1-\delta}}\right)
\end{equation}
Breaking the $n$-sum modulo $qq'p^{1-\delta}$ and application of Poisson summation yields,
\begin{equation}
\mathbf{T}=\sum_n\delta(n\equiv\overline{m4_qD_q}q'-\overline{m'4_{q'}D_{q'}}q\bmod qq'p^{1-\delta})\int_\BR \frac{1}{y}U(y)e\left(\frac{-nyR}{qq'p^{1-\delta}}\right)dy.
\end{equation}
Repeated integration by parts shows that we get arbitrary saving in powers of $t$ when $n> qq'p^{1-\delta}t^\varepsilon/R$. Therefore,
\begin{equation}
\begin{split}
S_{2,\delta,L_1}(N,C,R,\tau)\ll &\frac{K}{NC^2p}B(C,\tau)^2\sum_{n\ll\frac{C^2p^{1-\delta}t^{\varepsilon}}{R}}\underset{\begin{subarray}{c}C<q\leq2C\\L_1=(4D,q)p^{1-\delta}\end{subarray}}{\sum}\underset{\begin{subarray}{c}(m,q)=1\\(m,p)=p^\delta\\1\leq|m|\ll \frac{qpt^{1+\epsilon}}{N}\end{subarray}}{\sum}\\&
\times\underset{\begin{subarray}{c}C<q'\leq2C\\L_1=(4D,q')p^{1-\delta}\end{subarray}}{\sum}\underset{\begin{subarray}{c}(m',q')=1\\(m',p)=p^\delta\\1\leq|m'|\ll \frac{q'pt^{1+\epsilon}}{N}\end{subarray}}{\sum} \delta(n\equiv\overline{m4_qD_q}q'-\overline{m'4_{q'}D_{q'}}q\bmod qq'p^{1-\delta}) + O(t^{-2017})
\end{split}
\end{equation}
When $n=0$, the congruence condition forces $q=q'$. Moreover when $m$ is fixed, $m'$ is determined up to $p^\delta t^{1+\varepsilon}/N$. In the case $n\neq0$, fixing $n, q, q'$ determines $m$ up to $pt^{1+\varepsilon}/N$ and $m'$ up to $p^\delta t^{1+\varepsilon}/N$. That gives us,
\begin{equation*}
S_{2,\delta,L_1}(N,C,R,\tau)\ll_D\frac{Kt^\varepsilon}{NC^2p}B(C,\tau)^2\left[\frac{C^2p^{1-\delta}t^2}{N^2}+ \frac{C^4p^{2-2\delta}t^2}{RN^2}\right]=\frac{Kt^{2+\varepsilon}}{N^3p^\delta}B(C,\tau)^2\left[1+\frac{C^2p^{1-\delta}}{R}\right].
\end{equation*}
Multiplying by $N^{1/2}K$, summing over $R$ dyadically and using (\ref{bctau}),
\begin{equation}
S_2(N,C)\ll_D \sum_{\delta\in\{0,1\}}\underset{\begin{subarray}{c}L_1|L\\ p^{1-\delta}|L_1\end{subarray}}{\sum}\frac{K^{3/2}t^{1+\varepsilon}p^{1/2-\delta}}{N}[K^{1/2}D^{1/2}+C]\times\left[\frac{K}{t^{1/2}K^{3/2}}+\frac{1}{t^{1/2}K^{5/2}}\left(\frac{N}{QCp}\right)^{3/2}\right]
\end{equation}
We assume $K>(N/p)^{1/2}$, so that $K^{1/2}D^{1/2}+C\asymp K^{1/2}D^{1/2}$. Multiplying by $N^{1/2}/K$ and summing over $C$ dyadically,
\begin{equation}\label{S_2}
\frac{S_2(N)}{N^{1/2}}\ll_{\varepsilon,D} t^{1/2+\varepsilon}p^{1/4}\left(\frac{K^{1/2}p^{1/4}}{N^{1/2}}+\frac{N^{1/4}}{K^{3/4}p^{1/2}}\right).
\end{equation}

\section{Application of Cauchy inequality and Poisson summation- II}

We recall that
\begin{equation*}
\begin{split}
S_{l,J}(N,C) =&\sum_{\delta\in\{0,1\}}\underset{\begin{subarray}{c}L_1|L\\ p^{1-\delta}|L_1\end{subarray}}{\sum}\frac{i^kN^{1/2-it}K}{2}\sum_{n\ll KDp^{1-\delta}t^\varepsilon} \frac{\lambda_{L_2}(n)}{n^{1/2}}\underset{\begin{subarray}{c}C<q\leq2C\\L_1=(4D,q)p^{1-\delta}\end{subarray}}{\sum}\underset{\begin{subarray}{c}(m,q)=1\\(m,p)=p^\delta\\1\leq|m|\ll \frac{qpt^{1+\epsilon}}{N}\end{subarray}}{\sum}\chi_{L_1}(-\overline{m}p^{\delta})\chi_{L_2}(-qp^{1-\delta})\\&\times\eta(L_2)\frac{1}{aqp}e\left(-n\frac{\overline{m4_qD_q}}{qp^{1-\delta}}\right)\int_{\R} \left(\frac{2\pi\sqrt{Nn}}{qp^{1-\delta/2}\sqrt{4_qD_q}}\right)^{-i\tau}\gamma\left(1 + i\tau\right) W_J(\tau) \mathcal{I}_{1,J}(q,m,\tau) d\tau + O(t^{-2015})
\end{split}
\end{equation*}
where
\begin{equation*}
\mathcal{I}_{1,J}(q,m,\tau) = \frac{c_4}{(t+\tau/2)^{1/2}K}\left(-\frac{(t+\tau/2)qp}{2\pi eNm}\right)^{3/2-i(t+\tau/2)}V\left(-\frac{(t+\tau/2)qp}{2\pi Nm}\right)\int_0^1 V\left(\frac{\tau}{2K} - \frac{(t+\tau/2)x}{Kma}\right)dx.
\end{equation*}

Swapping the $q,m$-sums with the $\tau$-integral, taking absolute values, applying Cauchy inequality to the $n$-sum and using the Ramanujan bound on average, we get
\begin{equation}
S_{1,J}(N,C)\ll \sum_{\delta\in\{0,1\}}\underset{\begin{subarray}{c}L_1|L\\ p^{1-\delta}|L_1\end{subarray}}{\sum}t^\varepsilon N^{1/2}K\underset{dyadic}{\sum_{1\leq R \ll KDp^{1-\delta}t^{\epsilon}}} R^{1/2}[S_{1,J,\delta,L_1}(N,C,R)]^{1/2}.
\end{equation}
where
\begin{equation}
\begin{split}
S_{1,J,\delta,L_1}(N,C,R)=&\sum_n \frac{1}{n}U\left(\frac{n}{R}\right)\bigg|\int_{\R} (\sqrt{nN})^{-i\tau}\gamma(1+i\tau) \underset{\begin{subarray}{c}C<q\leq2C\\L_1=(4D,q)p^{1-\delta}\end{subarray}}{\sum}\underset{\begin{subarray}{c}(m,q)=1\\(m,p)=p^\delta\\1\leq|m|\ll \frac{qpt^{1+\epsilon}}{N}\end{subarray}}{\sum}\frac{\chi_{L_1}(-\overline{m}p^{\delta})\chi_{L_2}(-qp^{1-\delta})}{a(qp)^{1-i\tau}p^{i\tau\delta/2}}\\&\times e\left(-n\frac{\overline{m4_qD_q}}{qp^{1-\delta}}\right)W_{J}(\tau)I_1(q,m,\tau) d\tau\bigg|^2
\end{split}
\end{equation}
Opening the absolute value squared,
\begin{equation}
\begin{split}
S_{1,J,\delta,L_1}(N,C,R)=&\int_\BR\int_\BR \sqrt{N}^{-i\tau+i\tau'}\gamma(1+i\tau)\overline{\gamma(1+i\tau')} W_J(\tau)W_J(\tau')\underset{\begin{subarray}{c}C<q\leq2C\\L_1=(4D,q)p^{1-\delta}\end{subarray}}{\sum}\underset{\begin{subarray}{c}(m,q)=1\\(m,p)=p^\delta\\1\leq|m|\ll \frac{qpt^{1+\epsilon}}{N}\end{subarray}}{\sum}\underset{\begin{subarray}{c}C<q'\leq2C\\L_1=(4D,q')p^{1-\delta}\end{subarray}}{\sum}\underset{\begin{subarray}{c}(m',q')=1\\(m',p)=p^\delta\\1\leq|m'|\ll \frac{q'pt^{1+\epsilon}}{N}\end{subarray}}{\sum}\\ & \frac{\chi_{L_1}(-\overline{m}p^{\delta})\chi_{L_2}(-qp^{1-\delta})}{a(qp)^{1-i\tau}p^{i\tau\delta/2}}\frac{\overline{\chi}_{L_1}(-\overline{m'}p^{\delta})\overline{\chi}_{L_2}(-q'p^{1-\delta})}{a'(q'p)^{1+i\tau'}p^{-i\tau'\delta/2}}I_1(q,m,\tau)\overline{I_1(q',m',\tau')}\times\mathbf{T}d\tau d\tau'
\end{split}
\end{equation}
where
\begin{equation*}
\textbf{T} = \sum_n n^{-1+\frac{-i\tau+i\tau'}{2}} U\left(\frac{n}{R}\right)e\left(-n\frac{\overline{m4_qD_q}}{qp^{1-\delta}}\right)e\left(n\frac{\overline{m'4_{q'}D_{q'}}}{q'p^{1-\delta}}\right).
\end{equation*}
We analyze $\mathbf{T}$ by breaking the $n$-sum modulo $qq'p^{1-\delta}$ and applying Poisson summation,
\begin{equation}
\mathbf{T}=R^{-i(\tau-\tau')/2}\sum_n \delta(n\equiv\overline{m4_qD_q}q'-\overline{m'4_{q'}D_{q'}}q\bmod qq'p^{1-\delta})U^\dagger\left(\frac{nR}{qq'p^{1-\delta}},\frac{-i(\tau-\tau')}{2}\right).
\end{equation}
Since $|\tau-\tau'|\ll (NK)^{1/2}t^{\epsilon}/Cp^{1/2}$, the bound on $U^{\dagger}$ gives arbitrary saving for $|n|\gg C(NK)^{1/2}p^{1/2-\delta}t^{\varepsilon}/R$. We therefore get
\begin{lem}
\begin{equation}
\begin{split}
S_{1,J,\delta,L_1}(N,C,R)\ll &\frac{K}{NC^2p}\underset{\begin{subarray}{c}C<q\leq2C\\L_1=(4D,q)p^{1-\delta}\end{subarray}}{\sum}\underset{\begin{subarray}{c}(m,q)=1\\(m,p)=p^\delta\\1\leq|m|\ll \frac{qpt^{1+\epsilon}}{N}\end{subarray}}{\sum}\underset{\begin{subarray}{c}C<q'\leq2C\\L_1=(4D,q')p^{1-\delta}\end{subarray}}{\sum}\underset{\begin{subarray}{c}(m',q')=1\\(m',p)=p^\delta\\1\leq|m'|\ll \frac{q'pt^{1+\epsilon}}{N}\end{subarray}}{\sum}\sum_{n\ll C(NK)^{1/2}p^{1/2-\delta}t^\varepsilon/R}\\&\delta(n\equiv\overline{m4_qD_q}q'-\overline{m'4_{q'}D_{q'}}q\bmod qq'p^{1-\delta})|\mathfrak{K}| + O(t^{-2017}).
\end{split}
\end{equation}
where
\begin{equation}
\begin{split}
\mathfrak{K} = \underset{\R^2}{\int\int}(NR)^{-i\tau/2+i\tau'/2}&\gamma(1+i\tau)\overline{\gamma(1+i\tau')}\frac{1}{q^{-i\tau}q'^{i\tau}p^{-i(\tau-\tau')(1-\delta/2)}}W_J(\tau)W_J(\tau') \mathcal{I}_1(q,m,\tau)\overline{\mathcal{I}_1(q',m',\tau')}\\&\times U^{\dagger}\left(\frac{nR}{qq'p^{1-\delta}}, \frac{-i(\tau-\tau')}{2}\right) d\tau d\tau'
\end{split}
\end{equation}
\end{lem}
Using the expression for $\mathcal{I}_1(q,m,\tau)$ as given in lemma \ref{lemmaI**}, we get the expression
\begin{equation}
\begin{split}
\mathfrak{K} = \frac{|c_4|^2}{K^2}\underset{\R^2}{\int\int} & \gamma(1+i\tau)\overline{\gamma(1+i\tau')}W_J(q,m.\tau)\overline{W_J(q',m',\tau')}\frac{(RN)^{-i\tau/2+i\tau'/2}}{q^{-i\tau}q'^{i\tau'}p^{-i(\tau-\tau')(1-\delta/2)}}\left(-\frac{(t+\tau/2)qp}{2\pi eNm}\right)^{-i(t+\tau/2)} \\ &\left(-\frac{(t+\tau'/2)q'p}{2\pi eNm'}\right)^{i(t+\tau'/2)}U^{\dagger}\left(\frac{nR}{qq'p^{1-\delta}}, -\frac{i\tau}{2} + \frac{i\tau'}{2}\right) d\tau d\tau' 
\end{split}
\end{equation}
where
\begin{equation*}
W_J(q,m,\tau) = \frac{1}{(t+\tau/2)^{1/2}}W_J(\tau)\left(-\frac{(t+\tau/2)qp}{2\pi eNm}\right)^{3/2} V\left(-\frac{(t+\tau/2)qp}{2\pi Nm}\right) \int_0^1 V\left(\frac{\tau}{2K} - \frac{(t+\tau/2)x}{Kma}\right) dx
\end{equation*}
Since $u^{3/2}V(u)\ll 1$ and $\tau\ll J\ll t^{1-\epsilon}$, it follows that
\begin{equation}
\frac{\partial}{\partial\tau}W_J(q,m,\tau) \ll \frac{1}{t^{1/2}|\tau|}
\end{equation}
We also note that the $x$-integral inside the expression of $W_J(q,m,\tau)$ contributes a factor of the size of its length, which is $\ll Kma/(t+\tau)$. Since $m\ll Cpt^{1+\varepsilon}/N$ and $\tau\ll t$, the contribution is $\ll KCpQt^{\varepsilon}/N$. Therefore $W_{J}(q,m,\tau)\ll K^{1/2}p^{1/2}C/t^{1/2}N^{1/2}$.
Like before, we analyze the integral $\mathfrak{K}$ in two cases, when $n=0$ and when $n\neq0$. For $n=0$, the congruence condition implies $q=q'$, and the bound on $U^{\dagger}$ gives arbitrary saving for $|\tau-\tau'|\gg t^{\epsilon}$. In this case,
\begin{equation*}
\mathfrak{K}\ll \frac{|c_4|^2}{K^2}\underset{|\tau|\ll (NK)^{1/2}/Cp^{1/2}}{\int}|\gamma(1+i\tau)|^2W_J(q,m,\tau)\int_{|\tau'-\tau|\ll t^{\varepsilon}}W_J(q,m',\tau')d\tau' d\tau \ll \frac{t^{\epsilon}Cp^{1/2}}{K^{1/2}N^{1/2}t} =: B^*(C,0)
\end{equation*}
When $n\neq0$, 
\begin{equation}
\begin{split}
U^{\dagger}\left(\frac{nR}{qq'p^{1-\delta}}, -\frac{i\tau}{2} + \frac{i\tau'}{2}\right) = \frac{c_5}{(\tau-\tau')^{1/2}} &U\left(\frac{(\tau-\tau')qq'p^{1-\delta}}{4\pi nR}\right)\left(\frac{(\tau-\tau')qq'p^{1-\delta}}{4\pi enR}\right)^{-i\tau/2+i\tau'/2} \\&+ O\left(\min\left\lbrace\frac{1}{|\tau-\tau'|^{3/2}}, \frac{C^3p^{3(1-\delta)/2}}{(|n|R)^{3/2}}\right\rbrace\right)
\end{split}
\end{equation}
for some absolute constant $c_5$.

Contribution of the error term towards $\mathfrak{K}$ is of the order of
\begin{equation*}
\frac{t^{\epsilon}}{K^2}\underset{[J,4J/3]^2}{\int\int} \frac{1}{t} \min\left\lbrace\frac{1}{|\tau-\tau'|^{3/2}}, \frac{C^3p^{3(1-\delta)/2}}{(|n|R)^{3/2}}\right\rbrace d\tau d\tau'
\end{equation*}
When the second term is smaller,
\begin{equation}
\frac{t^{\epsilon}}{K^2}\underset{\begin{subarray}{c} [J,4J/3]^2 \\ |\tau-\tau'|\ll |n|R/C^2p^{1-\delta}\end{subarray}}{\int\int} \frac{1}{t} \frac{C^3p^{3(1-\delta)/2}}{(|n|R)^{3/2}} d\tau d\tau' \ll \frac{1}{K^{3/2}t}\frac{N^{1/2}}{(|n|R)^{1/2}p^{\delta/2}}t^{\epsilon}.
\end{equation}
When the first term is smaller,
\begin{equation}
\begin{split}
\frac{t^{\epsilon}}{K^2}\underset{\begin{subarray}{c} [J,4J/3]^2 \\ |\tau-\tau'|\gg |n|R/C^2p^{1-\delta}\end{subarray}}{\int\int} \frac{1}{t} \frac{1}{|\tau-\tau'|^{3/2}} d\tau d\tau' & \ll \frac{t^{\epsilon}}{K^2t} \frac{Cp^{(1-\delta)/2}}{(|n|R)^{1/2}} \underset{[J,4J/3]^2}{\int\int}\frac{1}{|\tau-\tau'|^{1-\epsilon}}d\tau d\tau' \\ & \ll \frac{1}{K^{3/2}t}\frac{N^{1/2}}{(|n|R)^{1/2}p^{\delta/2}}t^{\varepsilon}.
\end{split}
\end{equation}
Therefore the error contribution (for $n\neq0$) is
\begin{equation*}
B^*(C,n) = \frac{1}{K^{3/2}t}\frac{N^{1/2}}{(|n|R)^{1/2}p^{\delta/2}}t^{\varepsilon}
\end{equation*}

We finally analyze the main term for the case $n\neq0$. Striling's formula is
\begin{equation*}
\Gamma(\sigma+i\tau) = \sqrt{2\pi}(i\tau)^{\sigma-1/2}e^{-\pi|\tau|/2}\left(\frac{|\tau|}{e}\right)^{i\tau}\left\lbrace 1+ O\left(\frac{1}{|\tau|}\right) \right\rbrace
\end{equation*}
as $|\tau|\rightarrow\infty$. That gives
\begin{equation}
\gamma(1+i\tau) = \left(\frac{|\tau|}{4\pi e}\right)^{i\tau} \Phi(\tau), \quad \textit{where } \Phi'(\tau)\ll \frac{1}{|\tau|}
\end{equation}

By Fourier inversion, we write
\begin{equation*}
\left(\frac{4\pi nR}{(\tau-\tau')qq'p^{1-\delta}}\right)^{1/2} U\left(\frac{(\tau-\tau')qq'p^{1-\delta}}{4\pi nR}\right) = \int_{\R} U^{\dagger}(r,1/2)e\left(\frac{(\tau-\tau')qq'p^{1-\delta}}{4\pi nR}r\right) dr
\end{equation*}
We conclude that for some constant $c_6$ (depending on the sign of $n$)
\begin{equation}\label{Kfinal}
\mathfrak{K} = \frac{c_6}{K^2}\left(\frac{qq'p^{1-\delta}}{|n|R}\right)^{1/2}\int_{\R}U^{\dagger}(r,1/2)\underset{\R^2}{\int\int} g(\tau,\tau')e(f(\tau,\tau'))d\tau d\tau' dr + O(B^*(C,n))
\end{equation}
where
\begin{equation*}
\begin{split}
2\pi f(\tau,\tau') = & \tau\log\left(\frac{\tau}{4\pi e}\right) - \tau'\log\left(\frac{\tau'}{4\pi e}\right) - \frac{(\tau-\tau')}{2}\log(RN) + \tau\log q - \tau'\log q' + (\tau-\tau')(1-\delta/2)\log p \\ & -(t+\tau/2)\log\left(-\frac{(t+\tau/2)qp}{2\pi eNm}\right) + (t+\tau'/2)\log\left(-\frac{(t+\tau'/2)q'p}{2\pi eNm'}\right) \\ &+\frac{(\tau-\tau')}{2}\log\left(\frac{(\tau-\tau')qq'p^{1-\delta}}{4\pi enR}\right) + \frac{(\tau-\tau')qq'p^{1-\delta}}{2nR}r
\end{split}
\end{equation*}
and
\begin{equation*}
g(\tau,\tau') = \Phi(\tau)\overline{\Phi(\tau')}W_J(q,m,\tau)W_J(q',m',\tau')
\end{equation*}
We intend to use the second derivative bound as given in Lemma \ref{2nd.der.lem}. For that, we need the following
\begin{equation*}
2\pi\frac{\partial^2}{\partial\tau^2}f(\tau,\tau') = \frac{1}{4}\left(\frac{4}{\tau} - \frac{1}{(t+\tau/2)} + \frac{2}{(\tau'-\tau)}\right), \quad 2\pi\frac{\partial^2}{\partial\tau'^2}f(\tau,\tau') = \frac{1}{4}\left(\frac{-4}{\tau'} + \frac{1}{(t+\tau'/2)} + \frac{2}{(\tau'-\tau)}\right)
\end{equation*}
and
\begin{equation*}
2\pi\frac{\partial^2}{\partial\tau'\partial\tau}f(\tau,\tau') = \frac{-1}{4}\left(\frac{2}{\tau'-\tau}\right)
\end{equation*}
Also, by explicit computation,
\begin{equation*}
4\pi^2\left[ \frac{\partial^2}{\partial\tau^2}f(\tau,\tau')\frac{\partial^2}{\partial\tau'^2}f(\tau,\tau') - \left(\frac{\partial^2}{\partial\tau'\partial\tau}f(\tau,\tau')\right)^2 \right] = -\frac{1}{2\tau\tau'} + O\left(\frac{1}{tJ}\right)
\end{equation*}
for $\tau,\tau'$ such that $g(\tau,\tau')\neq0$. So the conditions of lemma 4 of Munshi \cite{munshi} hold with $r_1=r_2=1/J^{1/2}$. To calculate the total variation of $g(\tau,\tau')$, recall that $\Phi'(\tau)\ll |\tau|^{-1}$ and $W'_J(q,m,\tau)\ll t^{-1/2}|\tau|^{-1}$. So $var(g)\ll t^{-1+\epsilon}$. So the double integral in (\ref{Kfinal}) over $\tau,\tau'$ is bounded by $O(Jt^{-1+\epsilon})$. Integrating trivially over $r$ using the rapid decay of the Fourier transform, we get that total contribution of the leading term in (\ref{Kfinal}) towards $\mathfrak{K}$ is bounded by
\begin{equation*}
O\left( \frac{1}{K^2} \frac{Cp^{(1-\delta)/2}}{(|n|R)^{1/2}} \frac{(NK)^{1/2}}{Cp^{1/2}}t^{-1+\epsilon}\right) = O(B^*(C,n))
\end{equation*}

Putting everything together, we get the final bound
\begin{equation*}
\begin{split}
S_{1,J,\delta,L_1}(N,C,R) & \ll_D \frac{t^{\epsilon}K}{NC^2p}\bigg[\frac{C^2p^{1-\delta}t^2}{N^2}B^*(C,0) + \frac{C^2p^{1-\delta}t^2}{N^2}\sum_{n\ll C(NK)^{1/2}p^{1/2-\delta}t^\varepsilon/R}B^*(C,n) \bigg] \\ & = \frac{t^{\epsilon}K}{NC^2p} \bigg[\frac{C^3tp^{3/2-\delta}}{N^{5/2}K^{1/2}} + \frac{C^{5/2}tp^{5/4-2\delta}}{(NK)^{5/4}R}\bigg]
\end{split}
\end{equation*}

That gives
\begin{equation*}
\begin{split}
S_{1,J}(N,C) & \ll_D t^{\epsilon}N^{1/2}K\sum_{\delta\in\{0,1\}}\underset{\begin{subarray}{c}L_1|L\\ p^{1-\delta}|L_1\end{subarray}}{\sum} \underset{dyadic}{\sum_{1\leq R\ll KDpt^{\epsilon}}} R^{1/2} \frac{K^{1/2}}{N^{1/2}Cp^{1/2}} \left[\frac{C^{3/2}t^{1/2}p^{3/4-\delta/2}}{N^{5/4}K^{1/4}} + \frac{C^{5/4}t^{1/2}p^{5/8-\delta}}{(NK)^{5/8}R^{1/2}} \right] \\ &\ll_D t^{\epsilon}K^{3/2}\left( \frac{K^{1/4}C^{1/2}t^{1/2}p^{3/4}}{N^{5/4}} + \frac{C^{1/4}t^{1/2}p^{1/8}}{(NK)^{5/8}} \right)
\end{split}
\end{equation*}
Multiplying by $N^{1/2}/K$ and summing over the dyadic ranges of $J\in\mathcal{J}$ and $C\ll Q$, we get
\begin{equation}\label{S_1}
\frac{S_1(N)}{N^{1/2}} \ll_D t^{1/2+\varepsilon}p^{1/4} \left( \frac{K^{1/2}p^{1/4}}{N^{1/2}} + \frac{1}{(Kp)^{1/4}} \right)
\end{equation}
Finally, from equations (\ref{S_2}) and (\ref{S_1}), it follows that for $N\ll t^{1+\epsilon}$ and $K\gg N^{1/2}/p^{1/2}$,
\begin{equation*}
\frac{S(N)}{N^{1/2}}\ll_D t^{1/2+\varepsilon}p^{1/4}\left(\frac{K^{1/2}p^{1/4}}{N^{1/2}}+\frac{N^{1/4}}{K^{3/4}p^{1/2}}+ \frac{K^{1/2}p^{1/4}}{N^{1/2}} + \frac{1}{K^{1/4}p^{1/4}} \right).
\end{equation*}
The optimal choice for $K$ occurs at $K=(N/p)^{2/3}$ and we get Proposition \ref{mainprop}.

\begin{appendices}
\section{Voronoi formula for CM holomorphic cusp forms of squarefree level}

Let $K=\BQ(\sqrt{-D})$ be an imaginary quadratic field and $\psi$ be a \Gros of $K$ defined mod $\mathfrak{f}$ and has weight $r$. Define the $q$-series 
\begin{equation*}
f_\psi(z)=\underset{\begin{subarray}{c}\fra \textit{ integral }\\ \textit{ coprime to } \mathfrak{f}\end{subarray}}{\sum}\psi(\fra)q^{\calN(\fra)} \quad (q=e^{2\pi iz}, Im(z)>0).
\end{equation*}
Let $\chi$ be the Dirichlet character of conductor $\calN\mathfrak{f}$ attached to $\mathfrak{f}$, and $\chi_K$ be the field character, $\chi_K(p)=\left(\frac{disc(K)}{p}\right)$ defined on odd primes $p$ coprime with $disc(K)$, and extended to $\BZ$ by linearity. Hecke proved that $f_\psi$ is a cusp form of weight $r+1$, level dividing $disc(K)\calN(\mathfrak{f})$ and nebetypus $\chi\chi_K$. Shimura pointed out that $f_\psi$ is indeed a newform if $\psi$ has conductor $\mathfrak{f}$. The cusp forms $f_\psi$ have CM by $\chi_K$. However it isn't obvious that these are all the cusp forms with CM in $S_{r+1}(\Gamma_1(N)$. Ribet used the theory of Galois representations to prove that $f$ has CM by an imaginary quadratic field $K$ if and only if it arises from a Hecke character of $K$. We shall use this correspondence in our proof of Voronoi summation formula.

We start by recalling that the ring of integers $\mathcal{O}_K$ of $K$ is $\BZ[\alpha]$ where 
\[
\alpha=
\begin{cases}
\sqrt{-D} & \textit{ when } -D\equiv2,3\mod4 \\
\frac{1+\sqrt{-D}}{2} & \textit{ when } -D\equiv1\mod4.
\end{cases}
\]
Let $\mu(K)$ be the roots of unity in $K$, which is also the group of units of $\calO_K$, and $\omega_K=|\mu(K)|$. Let $Cl(K)$ be the class group of $K$. The different $\frd$ of $K$ is the ideal $\sqrt{-D}\mathcal{O}_K$ with $\calN\frd = D$. The real character $\chi_D$ of conductor $D$ given by $\chi_D(n)= \left(\frac{-D}{n}\right)$ is called the field character. The value of $\chi_D(p)$ determines the splitting of a prime $p$ in $\mathcal{O}_K$; we have $p=\frp^2,\frp,\frp\bar{\frp}$ with $\frp\neq\bar{\frp}$ for $\chi_D(p)=0,-1,1$ respectively.

Let $p>2$ and $p=\frp\bar{\frp}$ be a split prime in $K$. Since $\calN\frp=p$, there is an isomorphism $\Theta:\mathcal{O}_K/\frp\rightarrow\BZ/p\BZ$. Let $g(x)=x^2+D$. By Dedekind's theorem, $g(x)\equiv h(x)h'(x)\mod p$, splits into two distinct (linear) factors modulo $p$. Let $h(x)=x-d$, and let $\frp$ be generated by $\{p, h(\alpha)\}$ over $\mathcal{O}_K$. When $-D\equiv2,3\mod4$, one possible isomorphism $\Theta$ is given by $\Theta(a+b\sqrt{-D})=a+bd\bmod p$, where $a,b\in\BZ$. When $-D\equiv1\mod4$, a possible isomorphism is $\Theta\left(\frac{a+b\sqrt{-D}}{2}\right)=(a+bd)\bar{2}\bmod p$, where $\bar{2}$ is the inverse of $2$ modulo $p$.

Let $\lambda(n)=\sum_{\calN\fra=n}\psi(\fra)$. Our aim is to find a Voronoi summation formula for
\begin{equation}
\sum_{n\geq1}\lambda(n)e\left(\frac{mn}{q}\right)V\left(\frac{n}{N}\right)
\end{equation}
where $(m,q)=1$ and $V$ is a compactly supported smooth function with bounded derivatives, supported away from $0$. This sum can be rewritten as a sum over the integral ideals,
\begin{equation}\label{psisum}
\sum_{\fra\subset\calO_K}\psi(\fra)e\left(\frac{m\calN\fra}{q}\right)V\left(\frac{\calN\fra}{N}\right).
\end{equation}
We break this sum into ideal classes. Each ideal class contains a prime factor $\frl$ of an odd split prime $\ell$ coprime with $p$ and $m$, and we denote this class by $[\frl]$. Such an $\frl$ exists since any imaginary field is a Galois extension of $\BQ$. Note that $\calN\frl=\ell$. There is a correspondence,
\begin{equation*}
\begin{split}
&[\frl]\-\leftrightarrow\frl/\units\\
&\fra\mapsto\fra\frl=(\gamma).
\end{split}
\end{equation*}
Then $\psi(\fra)=\psi(\gamma)/\psi(\frl)$ and $\calN\fra=\calN\gamma/\ell$. Expression (\ref{psisum}) can be written as,
\begin{equation}\label{idealsum}
\begin{split}
&\sum_{[\frl]\-\in Cl(K)}\frac{1}{\psi(\frl)}\sum_{\gamma\in\frl/\mu(K)}\psi(\gamma)e\left(\frac{m\calN\gamma}{q\ell}\right)V\left(\frac{\calN\gamma}{N\ell}\right)\\
=&\frac{1}{\omega_K}\sum_{[\frl]\-\in Cl(K)}\frac{1}{\psi(\frl)}\sum_{\gamma\in\frl}\psi(\gamma)e\left(\frac{m\calN\gamma}{q\ell}\right)V\left(\frac{\calN\gamma}{N\ell}\right).
\end{split}
\end{equation}
For brevity of notation, we define
\begin{equation}\label{frlsum}
S_\frl=\sum_{\gamma\in\frl}\psi(\gamma)e\left(\frac{m\calN\gamma}{q\ell}\right)V\left(\frac{\calN\gamma}{N\ell}\right).
\end{equation}
In our approach, we need an explicit choice of basis for the $\BZ$-lattice formed by the ideal $\frl$. Since the norm $\calN$ gives a quadratic form as opposed to a linear form, a good choice of basis is crucial in simplifying the calculations. Let $m(x)$ be the minimal polynomial of $\alpha$ and $\ell$ be an odd split prime, say $(\ell)=\frl\frl'$. by Dedekind's theorem, $m(x)\equiv m_1(x)m_2(x)\bmod\ell$ factors into linear polynomials and $\frl$ is generated by $\{\ell, m_1(\alpha)\}$ over $\calO_K$.

\begin{lem}
$\{\ell, m_1(\alpha)\}$ is a $\BZ$-basis of $\frl$.
\end{lem}
\begin{proof}
Let $\Gamma$ be a lattice generated by $\{\ell, m_1(\alpha)\}$ over $\BZ$. Then $\Gamma=\{u\ell+vm_1(\alpha) | u,v\in\BZ\}$. Let $m_1(x)=x-c_\ell$. Then the lattice is given by $\{(u\ell-vd_\ell)+v\alpha | u,v\in\BZ\}$. Equivalently, $\Gamma=\{x+y\alpha | x+yc_\ell \textit{ is divisible by } \ell\}$. 

Next, let $m(x)=x^2-ax-b$ be the minimal polynomial of $\alpha$. Since $\frl$ is generated by $\ell$ and $m_1(\alpha)$ over $\calO_K$, an element $\beta\in\frl$ can be written as a linear combination,
\begin{equation*}
\begin{split}
\beta &= (a_1+b_1\alpha)\ell + (a_2+b_2\alpha)(\alpha-c_\ell)\\
&= (a_1\ell-a_2c_\ell+b_2b)+(b_1\ell+a_2-b_2c_\ell+ab_2)\alpha
\end{split}
\end{equation*}
Then $(a_1\ell-a_2c_\ell+b_2b)+c_\ell(b_1\ell+a_2-b_2c_\ell+ab_2)=(a_1+b_1c_\ell)\ell-(c_\ell^2-ac_\ell-b)b_2\equiv0\bmod\ell$. Therefore $\beta\in\Gamma$.
\end{proof}
\begin{notation}
Since $\ell$ splits, $x^2+D=h_\ell(x)h_\ell'(x)$ splits into linear factors modulo $\ell$. Say $\frl$ is generated by $\{\ell,h_\ell(\alpha)\}$ over $\calO_K$. We define $d_\ell$ to be integer such that $h_\ell(x)=x-d_\ell$.
\end{notation}
\begin{cor}
The lattice of ideal $\frl$ is given by,
\begin{equation}
\Gamma_\ell=\bigg\lbrace a+b\sqrt{-D}\mid a, b\in\BZ, a+bd_\ell \textit{ is divisible by } \ell\bigg\rbrace \quad \textit{ when } -D\equiv 2, 3\bmod4
\end{equation}
and
\begin{equation}
\Gamma_\ell=\bigg\lbrace\frac{a+b\sqrt{-D}}{2}\mid a, b\in\BZ, a\equiv b\bmod2, a+bd_\ell \textit{ is divisible by } \ell\bigg\rbrace \quad \textit{ when } -D\equiv1\bmod4.
\end{equation}
\end{cor}

Using the definition of $\psi$,
\begin{equation*}
S_\frl=\frac{1}{\ell}\sum_{\xi\mod\ell}\sum_{\begin{subarray}{c}(c,f)\in\BZ^2\end{subarray}}\chi(c+fd)\left(\frac{c+f\sqrt{-D}}{\sqrt{c^2+f^2D}}\right)^re\left(\frac{m(c^2+f^2D)}{q\ell}\right)e\left(\frac{(c+fd_\ell)\xi}{\ell}\right)V\left(\frac{c^2+f^2D}{N\ell}\right)
\end{equation*}
when $-D\equiv2,3\mod4$, while it is
\begin{equation*}
\begin{split}
S_\frl=\frac{1}{\ell}\sum_{\xi\mod \ell}\sum_{\begin{subarray}{c}(c,f)\in\BZ^2\end{subarray}}&\chi((c+fd)\bar{2})\left(\frac{c+f\sqrt{-D}}{\sqrt{c^2+f^2D}}\right)^re\left(\frac{m(c^2+f^2D)}{4q\ell}\right)e\left(\frac{(c+fd_\ell)\xi}{\ell}\right)V\left(\frac{c^2+f^2D}{4N\ell}\right)\\&\times\left(1+e\left(\frac{c+f}{2}\right)\right)
\end{split}
\end{equation*}
when $-D\equiv1\mod4$. The analysis of the two cases is similar, so we only present the details for the case $-D\equiv3\mod4$. We break the $(c,f)$-sum modulo $qp\ell$ by changing variables $c\mapsto\beta+qp\ell c$ and $f\mapsto\gamma+qp\ell f$,
\begin{equation*}
\begin{split}
S_\frl=\frac{1}{\ell}\sum_{\xi\mod\ell}\sum_{\begin{subarray}{c}(c,f)\in\BZ^2\end{subarray}}\sum_{\beta,\gamma\mod qp\ell}&\chi(\beta+\gamma d)\left(\frac{(\beta+qp\ell c)+(\gamma+qp\ell f)\sqrt{-D}}{\sqrt{Norm}}\right)^re\left(\frac{m(\beta^2+\gamma^2D)}{q\ell}\right)\\
&\times e\left(\frac{(\beta+\gamma d_\ell)\xi}{\ell}\right)V\left(\frac{Norm}{N\ell}\right)
\end{split}
\end{equation*}
where $Norm=\calN[(\beta+qp\ell c)+(\gamma+qp\ell f)\sqrt{-D}]$. Application of Poisson summation formula to the $(c,f)$-sum followed by a change of variables gives,
\begin{equation}\label{arithint}
\begin{split}
S_\frl=&\sum_{\begin{subarray}{c}(c,f)\in\BZ^2\end{subarray}}\frac{1}{\ell}\sum_{\xi\mod \ell}\frac{1}{(qp\ell)^2}\sum_{\beta,\gamma\mod qp\ell}\chi(\beta+\gamma d)e\left(\frac{(\beta+\gamma d_\ell)\xi}{\ell}\right)e\left(\frac{m(\beta^2+\gamma^2D)}{q\ell}\right)e\left(\frac{c\beta+f\gamma}{qp\ell}\right) \\
&\times\int\int\left(\frac{z+y\sqrt{-D}}{\sqrt{z^2+y^2D}}\right)^r V\left(\frac{z^2+y^2D}{N\ell}\right)e\left(\frac{-cz-fy}{qp\ell}\right)dzdy\\
=&\sum_{(c,f)\in\BZ^2}\calA(mp,(c,f);q)\CJ((c,f);q).
\end{split}
\end{equation}
Here $\calA(mp,(c,f);q)$ is the `arithmetic part',
\begin{equation}\label{arith}
\calA(mp,(c,f);q)=\frac{1}{\ell}\sum_{\xi\mod\ell}\frac{1}{(qp\ell)^2}\sum_{\beta,\gamma\mod qp\ell}\chi(\beta+\gamma d)e\left(\frac{(\beta+\gamma d_\ell)\xi}{\ell}\right)e\left(\frac{m(\beta^2+\gamma^2D)}{q\ell}\right)e\left(\frac{c\beta+f\gamma}{qp\ell}\right)
\end{equation}
and $\CJ((c,f);q)$ is the `analytic part',
\begin{equation}
\CJ((c,f);q)= \int\int\left(\frac{z+y\sqrt{-D}}{\sqrt{z^2+y^2D}}\right)^r V\left(\frac{z^2+y^2D}{N\ell}\right)e\left(\frac{-cz-fy}{qp\ell}\right)dzdy.
\end{equation}
\subsection{Analytic part}
A change of variables $z\mapsto z\sqrt{N\ell}, y\mapsto y\sqrt{N\ell/D}$ followed by converting to polar coordinates shows,
\begin{equation*}
\begin{split}
\CJ((c,f);q)&=\frac{N\ell}{2\sqrt{D}}\int_0^\infty\int_0^{2\pi}e^{ir\theta}V(R)e\left(\frac{-\sqrt{NR}}{qp\sqrt{\ell D}}(c\sqrt{D}\cos\theta+f\sin\theta)\right)d\theta dR\\
&=\frac{N\ell e^{-ir\varphi}}{2\sqrt{D}}\int_0^\infty V(R)\int_0^{2\pi}e^{ir\theta}e\left(\frac{-\sqrt{NR(c^2D+f^2)}}{qp\sqrt{\ell D}}\sin\theta\right)d\theta dR,
\end{split}
\end{equation*}
where $\tan\varphi=c\sqrt{D}/f$. The $\theta$-integral gives a $J$-Bessel function \cite{grry6},
\begin{equation}\label{analytic}
\CJ((c,f);q)=\frac{N\ell e^{-ir\varphi}(-1)^r\pi}{\sqrt{D}}\int_0^\infty V(R)J_r\left(\frac{2\pi\sqrt{NR(c^2D+f^2)}}{qp\sqrt{\ell D}}\right)dR.
\end{equation}

\subsection{Arithmetic part}
We separate the $\beta$ and $\gamma$ sums, compute those separately and finally take their product. By the Gauss formula,
\begin{equation*}
\chi(a)=\frac{1}{\tau(\bar{\chi})}\sum_{b\in\BF_{p}}\bar{\chi}(b)e\left(\frac{ab}{p}\right).
\end{equation*}
Then,
\begin{equation*}
\begin{split}
\calA(mp,(c,f);q)=&\frac{1}{\tau(\bar{\chi})}\sum_{b\mod p}\bar{\chi}(b)\frac{1}{\ell}\sum_{\xi\mod\ell}\frac{1}{(qp\ell)^2}\sum_{\beta,\gamma\mod qp\ell}e\left(\frac{(\beta+\gamma d)b}{p}\right)e\left(\frac{(\beta+\gamma d_\ell)\xi}{\ell}\right)\\&\times e\left(\frac{m(\beta^2+\gamma^2D)}{q\ell}\right)e\left(\frac{c\beta+f\gamma}{qp\ell}\right)\\
=&\frac{1}{\tau(\bar{\chi})}\sum_{b\mod p}\bar{\chi}(b)\frac{1}{\ell}\sum_{\xi\mod \ell}\calB(\xi, b, mp, c; q)\calB(d_\ell\xi, bd, mpD, f;q)
\end{split}
\end{equation*}
with
\begin{equation}\label{calB}
\calB(\xi, b, mp, c; q) = \frac{1}{qp\ell}\sum_{\beta\mod qp\ell}e\left(\frac{mp\beta^2+(c+qp\xi+bq\ell)\beta}{qp\ell}\right).
\end{equation}
We need a lemma about Gauss sums to compute (\ref{calB}). When $a$ is an odd integer, we define
\begin{equation*}
\varepsilon_a=
\begin{cases}
1 \quad \textit{ if } a\equiv1\bmod4\\
i \quad \textit{ if } a\equiv3\bmod4.
\end{cases}
\end{equation*}

\begin{lem}\label{gausslemma}
Let $$g(a,b,c)=\sum_{\beta\mod c}e\left(\frac{a\beta^2+b\beta}{c}\right)$$ be a Gauss sum with $(a,c)=1$ and $c>0$. If $c$ is even, let $c_1=c/2$. Then,
\begin{equation}
g(a,b,c)=
\begin{cases}
e\left(\frac{-\overline{4a}b^2}{c}\right)\left(\frac{a}{c}\right)\varepsilon_c\sqrt{c} \quad & \textit{ if } c \textit{ is odd},\\
2\delta(2\nmid b)e\left(\frac{-\overline{8a}b^2}{c_1}\right)\left(\frac{2a}{c_1}\right)\varepsilon_{c_1}\sqrt{c_1} \quad & \textit{ if } c\equiv2\bmod4,\\
\delta(2|b)e\left(\frac{-\overline{a}b^2}{4c}\right)\left(\frac{c}{a}\right)\frac{1+i}{\varepsilon_a}\sqrt{c} \quad & \textit{ if } 4|c.
\end{cases}
\end{equation}
\end{lem}
\begin{proof}
We use the following result from \cite{gsums98},
\begin{equation}\label{gauss}
g(a,0,c)=
\begin{cases}
\left(\frac{a}{c}\right)\varepsilon_c\sqrt{c} \quad & \textit{ if } c \textit{ is odd},\\
\quad 0 \quad & \textit{ if } c\equiv2\bmod4,\\
\left(\frac{c}{a}\right)\frac{1+i}{\varepsilon_a}\sqrt{c} \quad & \textit{ if } 4\mid c.
\end{cases}
\end{equation}
When $c$ is odd,
\begin{equation*}
g(a,b,c)=\sum_{\beta\bmod c}e\left(\frac{a(\beta^2+2\times\overline{2a}b\beta)}{c}\right)=e\left(\frac{-\overline{4a}b^2}{c}\right)\sum_{\beta\bmod c}e\left(\frac{a\beta^2}{c}\right).
\end{equation*}
We get the second equality above by completing squares and changing variables. Using (\ref{gauss}),
\begin{equation*}
g(a,b,c)= e\left(\frac{-\overline{4a}b^2}{c}\right)\left(\frac{a}{c}\right)\varepsilon_c\sqrt{c}.
\end{equation*}
When $c$ is even, say $2^k||c$, we let $c_k=c/2^k$. Using reciprocity,
\begin{equation*}
g(a,b,c)=\sum_{\gamma\bmod2^k}e\left(\frac{ac_k\gamma^2+b\gamma}{2^k}\right)\times\sum_{\beta\bmod c_k}e\left(\frac{2^ka\beta^2+b\beta}{c_k}\right).
\end{equation*}
When $k=1$, the $\gamma$-sum equals $2\delta(2\nmid b)$. When $k>1$, a change of variables $\gamma\mapsto\gamma+2^{k-1}$ shows that $g(a,b,c)=0$ unless $2|b$. When $k=1$, we follow the previous computation to get
\begin{equation*}
g(a,b,c)=2\delta(2\nmid b)e\left(\frac{-\overline{8a}b^2}{c_1}\right)\left(\frac{2a}{c_1}\right)\varepsilon_{c_1}\sqrt{c_1}.
\end{equation*}
When $k>1$, we let $b_1=b/2$ so that
\begin{equation*}
g(a,b,c)=\delta(2|b)\sum_{\beta\bmod c}e\left(\frac{a\beta^2+2b_1\beta}{c}\right)=\delta(2|b)e\left(\frac{-\overline{a}b^2}{4c}\right)\sum_{\beta\bmod c}e\left(\frac{a\beta^2}{c}\right).
\end{equation*}
Using (\ref{gauss}),
\begin{equation*}
g(a,b,c)=\delta(2|b)e\left(\frac{-\overline{a}b^2}{4c}\right)\left(\frac{c}{a}\right)\frac{1+i}{\varepsilon_a}\sqrt{c}.
\end{equation*}
\end{proof}
Let $c_\xi=c+qp\xi$ and $f_{d_\ell\xi}=f+qpd_\ell\xi$. Let $g=(q,D)$, $D_q=D/g$ and $q_D=q/g$. If $q$ is even, let $q_1=q/2$. When $g|f$, we define $f_g=f/g$ and $f_{gd_\ell\xi}=f_{d_\ell\xi}/g$. Let $g_p=(m,p)$ and $m_p=m/g_p$.

\subsubsection{case 1: $(p,q)=1$} Changing variables $\beta\mapsto\beta+q\ell$ in (\ref{calB}) shows that $p|(c+bq\ell)$, otherwise the sum is zero. Since $(p,q)=1$, $b$ is determined modulo $p$, $b\equiv-\overline{q\ell}c\bmod p$. Then,
\begin{equation*}
\calB(\xi, b, mp, c; q) = \delta(b\equiv-\overline{q\ell}c\bmod p)\frac{1}{q\ell}g(m,\overline{p}(c+qp\xi),q\ell).
\end{equation*}
Next, we notice that $\calB(d_\ell\xi, bd, mpD, f; q)=0$ unless $g|f$. Then $\calB(d_\ell\xi, bd, mpD, f; q)=\calB(d_\ell\xi, bd, mpD_q, f_g; q_D)\times\delta(g\mid f)$. Following the calculations as above,
\begin{equation*}
\calB(d_\ell\xi, bd, mpD_q, f_g; q_D)=\delta(bd\equiv-\overline{q_D\ell}f_g\bmod p)\delta(g\mid f)\frac{1}{q_D}g(mD_q,\overline{p}(f_g+q_Dpd_\ell\xi),q_D\ell).
\end{equation*}
We take their product to calculate $\calA(mp,(c,f);q)$. There is no $b$-sum since $b$ is determined $\bmod p$.
\begin{equation}
\begin{split}
\calA(mp,(c,f);q)=&\frac{1}{\tau(\overline{\chi})}\chi(-q\ell)\overline{\chi}(c)\delta(cd-f\equiv0\bmod p){\delta(g|f)}\\
&\times\frac{1}{\ell}\sum_{\xi\bmod\ell}\frac{g}{q^2\ell^2}g(m,\overline{p}(c+qp\xi),q\ell)g(mD_q,\overline{p}(f_g+q_Dpd_\ell\xi),q_D\ell).
\end{split}
\end{equation}
Using lemma \ref{gausslemma} and the fact that $d_\ell^2+D\equiv0\bmod\ell$, we see that there is no $\xi^2$ terms in the exponential. To be explicit, when $q$ is odd,
\begin{equation*}
\begin{split}
g(m,\overline{p}(c+qp\xi),q\ell)g(mD_q,\overline{p}(f_g+q_Dpd_\ell\xi),q_D\ell)=&e\left(\frac{-\overline{4mp^2D_q}[D(c+qp\xi)^2+(f+qpd_\ell\xi)^2]}{gq\ell}\right)\\&\times\left(\frac{m}{g}\right)\left(\frac{D_q}{q_D\ell}\right)\varepsilon_{q\ell}\varepsilon_{q_D\ell}\frac{q\ell}{\sqrt{g}}.
\end{split}
\end{equation*}
Expanding the squares shows that the coefficient of $\xi^2$ is $q^2p^2(D+d_\ell^2)$, which is divisible by $gq\ell$. Executing the sum over $\xi\bmod\ell$, we get that when $q$ is odd,
\begin{equation}\label{arith1qodd}
\begin{split}
\calA(mp,(c,f);q)=&\frac{1}{\tau(\overline{\chi})}\chi(-q\ell)\overline{\chi}(c)\delta(cd-f\equiv0\bmod p)\frac{\delta(g|f)\sqrt{g}}{q\ell}\delta(\ell\mid(cd_\ell-f))\\&\times\left(\frac{m}{g}\right)\left(\frac{D_q}{q_D\ell}\right)\varepsilon_{q\ell}\varepsilon_{q_D\ell}e\left(\frac{-\overline{4mp^2D_q}(c^2D+f^2)}{gq\ell}\right).
\end{split}
\end{equation}
Similarly, when $2||q$,
\begin{equation*}
\begin{split}
g(m,\overline{p}(c+qp\xi),q\ell)g(mD_q,\overline{p}(f_g+q_Dpd_\ell\xi),q_D\ell)=&e\left(\frac{-\overline{8mp^2D_q}[D(c+qp\xi)^2+(f+qpd_\ell\xi)^2]}{gq_1\ell}\right)\\&\times\left(\frac{2m}{g}\right)\left(\frac{D_q}{q_{1D}\ell}\right)\varepsilon_{q_1\ell}\varepsilon_{q_{1D}\ell}\frac{2q\ell}{\sqrt{g}}\delta(2\nmid c,f).
\end{split}
\end{equation*}
Therefore for $2||q$,
\begin{equation}\label{arith1q2}
\begin{split}
\calA(mp,(c,f);q)=&\frac{1}{\tau(\overline{\chi})}\chi(-q\ell)\overline{\chi}(c)\delta(cd-f\equiv0\bmod p)\frac{2\delta(g|f)\sqrt{g}}{q\ell}\delta(\ell\mid(cd_\ell-f))\\&\times\left(\frac{2m}{g}\right)\left(\frac{D_q}{q_{1D}\ell}\right)\varepsilon_{q_1\ell}\varepsilon_{q_{1D}\ell}\delta(2\nmid c,f)e\left(\frac{-\overline{8mp^2D_q}(c^2D+f^2)}{gq_1\ell}\right).
\end{split}
\end{equation}
Finally, when $4|q$,
\begin{equation*}
\begin{split}
g(m,\overline{p}(c+qp\xi),q\ell)g(mD_q,\overline{p}(f_g+q_Dpd_\ell\xi),q_D\ell)=&e\left(\frac{-\overline{mp^2D_q}[D(c+qp\xi)^2+(f+qpd_\ell\xi)^2]}{4gq\ell}\right)\\&\times\left(\frac{g}{m}\right)\left(\frac{q_{D}\ell}{D_q}\right)\frac{2i}{\varepsilon_{m}\varepsilon_{mD_q}}\frac{q\ell}{\sqrt{g}}\delta(2\mid c,f).
\end{split}
\end{equation*}
Therefore for $4|q$,
\begin{equation}\label{arith1q4}
\begin{split}
\calA(mp,(c,f);q)=&\frac{1}{\tau(\overline{\chi})}\chi(-q\ell)\overline{\chi}(c)\delta(cd-f\equiv0\bmod p)\frac{\delta(g|f)\sqrt{g}}{q\ell}\delta(\ell\mid(cd_\ell-f))\\&\times\left(\frac{g}{m}\right)\left(\frac{q_{D}\ell}{D_q}\right)\frac{2i}{\varepsilon_{m}\varepsilon_{mD_q}}\delta(2\mid c,f)e\left(\frac{-\overline{mp^2D_q}(c^2D+f^2)}{4gq\ell}\right).
\end{split}
\end{equation}

\begin{rem}
Recall that $\ell$ split in $K$ as $(\ell)=\frl\frl'$ and $x^2+D\equiv(x-d_\ell)(x+d_\ell)\bmod\ell$. We had assumed $\frl$ to be the ideal generated by $\{\ell,\sqrt{-D}-d_\ell\}$, so that the associated lattice is $\Gamma_\frl=\{(c,f)\mid c+fd_\ell\equiv0\bmod\ell\}$. Therefore the condition $\delta(\ell|(cd_\ell-f))$ implies $(f,c)\in\Gamma_{\frl'}$, the lattice associated with the ideal $\frl'$. Note that $[\frl']=[\frl]\-$ in the class group. Similarly, $p$ splits as $(p)=\frp\frp'$ with $\frp$ generated by $\{p,\sqrt{-D}-d\}$. Therefore the condition $\delta(p\mid(cd-f))$ implies $(f,c)\in\Gamma_{\frp'}$. 

Next, we note that $g|D$, so that the primes dividing $g$ ramify. $(g)$ factors as $(g)=\frg^2$ where $\frg$ is a squarefree product of prime ideals. The condition $g|f$ means that $(f,c)\in\Gamma_{\frg}=\{(f,c)\in\BZ^2\mid g|f\}$. Though we cannot use Dedekind's theorem to find generators of the ramified primes, it is easy to argue that $\Gamma_\frg$ is the lattice associated with $\frg$. The correspondence is given by 
\begin{equation}
\begin{split}
\Gamma_\frg&\leftrightarrow\frg\\
(f,c)&\mapsto f+c\sqrt{-D}.
\end{split}
\end{equation}
$\Gamma_\frg$ is closed under multiplication by $\calO_K$; given $(f,c)\in\Gamma_\frg$ and $(a,b)\in\calO_K$, $(f+c\sqrt{-D})(a+b\sqrt{-D})=(fa-bcD)+(fb+ca)\sqrt{-D}\in\Gamma_\frg$. Moreover $\Gamma_\frg$ strictly contains $\Gamma_g=\{(c,f)\in\BZ^2\mid g|f,c\}$, the lattice associated with the ideal $(g)$. This proves that $\Gamma_\frg$ is the lattice associated with $\frg$. Finally, we notice that $2$ also ramifies since $|disc(K)|=4D$, say $(2)=\frp_2^2$. The condition $2\nmid c,f$ in the case $2||q$ is the same as $f+c\sqrt{-D}\in\frp_2$, while the condition $2|c,f$ in the case $4|q$ is the same as $f+c\sqrt{-D}\in(2)$.

Therefore the condition $\delta(cd-f\equiv0\bmod p)\delta(\ell\mid(cd_\ell-f))\delta(g|f)$ is the same as $f+c\sqrt{-D}\in \frl'\frp'\frg$. Hence $c^2D+f^2$ is divisible by $\ell pg$.
\end{rem}

\subsubsection{case 2: $p\mid q$} In this case, $(m,p)=1$ and $\calB(\xi,b,mp,c;q)=0$ unless $p|c$. Let $c_p=c/p, c_{p\xi}=c_\xi/p$ and $q_p=q/p$. Then $\calB(\xi,b,mp,c;q)=\calB(\xi,b,m,c_p;q_p)$. We have,
\begin{equation}
\calB(\xi,b,m,c_p;q_p)=\frac{1}{q\ell}g(m,(c_{p\xi}+bq_p\ell),q\ell).
\end{equation}
Similarly $\calB(d_\ell\xi,bd,mpD,f;q)=\calB(d_\ell\xi,bd,mD_q,f_{pg};q_{pD})\delta(g|f)$. We have
\begin{equation}
\calB(d_\ell\xi,bd,mD_q,f_{pg};q_{pD})\delta(g|f)=\delta(g|f)\frac{\delta(g|f)}{q_D\ell}g(mD_q,(f_{pg\xi}+bdq_{pD}\ell),q_D\ell).
\end{equation}
We evaluate the gauss sums using lemma \ref{gausslemma}. The exponentials are quadratic in $b$. On taking the product of $\calB(\xi,b,m,c_p;q_p)$ and $\calB(d_\ell\xi,bd,mD_q,f_{pg};q_{pD})$, we will see that the $b^2$ term vanishes. To be explicit, we present the details of the case when $q$ is odd.
\begin{equation*}
\begin{split}
g(m,\overline{g_\ell}(c_{p\xi}+bq_p\ell),q\ell)g(mD_q,(f_{pg\xi}+bdq_{pD}\ell),q_D\ell)=&e\left(\frac{-\overline{4mD_q}[D(c_{p\xi}+bq_p\ell)^2+(f_{pd_\ell\xi}+bdq_p\ell)^2]}{gq\ell}\right)\\
&\times\left(\frac{m}{g}\right)\left(\frac{D_q}{q_D\ell}\right)\varepsilon_{q\ell}\varepsilon_{q_D\ell}\frac{q\ell}{\sqrt{g}}.
\end{split}
\end{equation*}
Expanding the squares, the coefficient of $b^2$ is $(D+d^2)q_p^2\ell^2$, which is divisible by $gq\ell$. One can check that the corresponding linear $b$-term is $e(-\overline{2mD}(Dc_{p}+f_{p}d)b/p)$. The sum over $b$ gives,
\begin{equation}
\frac{1}{\tau(\overline{\chi})}\sum_{b\bmod p}\overline{\chi}(b)e\left(\frac{-\overline{2mD}(Dc_{p}+f_{p}d)b}{p}\right)=\overline{\chi}(-4md)\chi(c_pd-f_p).
\end{equation}
The sum over $\xi\bmod\ell$ is executed exactly as before. Putting it all together, when $q$ is odd,
\begin{equation}\label{arithpqodd}
\begin{split}
\calA(m,(c_p,f_p);q_p)=&\overline{\chi}(-2md)\chi(c_pd-f_p)\frac{\delta(g|f)\sqrt{g}}{q\ell}\delta(p\mid c,f)\delta(\ell\mid(cd_\ell-f))\\&\times\left(\frac{m}{g}\right)\left(\frac{D_q}{q_D\ell}\right)\varepsilon_{q\ell}\varepsilon_{q_D\ell}e\left(\frac{-\overline{4mD_q}(c^2D+f^2)}{gqp^2\ell}\right).
\end{split}
\end{equation}
Similarly when $2||q$,
\begin{equation}\label{arithpq2}
\begin{split}
\calA(m,(c_p,f_p);q_p)=&\overline{\chi}(-2md)\chi(c_pd-f_p)\frac{2\delta(g|f)\sqrt{g}}{q\ell}\delta(p\mid c,f)\delta(\ell\mid(cd_\ell-f))\\&\times\left(\frac{2m}{g}\right)\left(\frac{D_q}{q_{1D}\ell}\right)\varepsilon_{q_1\ell}\varepsilon_{q_{1D}\ell}\delta(2\nmid c,f)e\left(\frac{-\overline{8mD_q}(c^2D+f^2)}{gq_1p^2\ell}\right).
\end{split}
\end{equation}
Finally when $4|q$,
\begin{equation}\label{arithpq4}
\begin{split}
\calA(m,(c_p,f_p);q_p)=&\overline{\chi}(-2md)\chi(c_pd-f_p)\delta(p\mid c,f)\frac{\delta(g|f)\sqrt{g}}{q\ell}\delta(\ell\mid(cd_\ell-f))\\&\times\left(\frac{g}{m}\right)\left(\frac{q_{D}\ell}{D_q}\right)\frac{2i}{\varepsilon_{m}\varepsilon_{mD_q}}\delta(2\mid c,f)e\left(\frac{-\overline{mD_q}(c^2D+f^2)}{4gqp^2\ell}\right).
\end{split}
\end{equation}
\begin{rem}
The condition $p|c, f$ corresponds to $(c,f)\in\Gamma_p=\{(c,f)\in\BZ^2\mid p \textit{ divides } c,f\}$, which is the lattice associated to the ideal $(p)$. Therefore the condition $\delta(p\mid c,f)\delta(\ell\mid(cd_\ell-f))\delta(g|f)$ is the same as $f+c\sqrt{-D}\in \frl'(p)\frg$. Therefore $c^2D+f^2$ is divisible by $\ell p^2g$.
\end{rem}

We have thus computed explicitly the analytic and arithmetic parts. We can get an explicit Voronoi formula by putting together the expressions (\ref{analytic}), (\ref{arith1qodd}), (\ref{arith1q2}), (\ref{arith1q4}), (\ref{arithpqodd}), (\ref{arithpq2}) and (\ref{arithpq4}). We shall update the article in the next version of the paper and show that our Voronoi formula matches with that in \cite{kmv, boming17}.


\end{appendices}

\subsection*{Acknowledgments} I would like to thank Prof. Ritabrata Munshi for suggesting the problem, and Prof. Roman Holowinsky for many insightful discussions and encouragement.

\bibliographystyle{amsplain}
\bibliography{ref}

\address{Department of Mathematics, The Ohio State University, 100 Math Tower,
231 West 18th Avenue, Columbus, OH 43210-1174}\\
Email address: \email{aggarwal.78@osu.edu}

\end{document}